\documentclass{amsart}

\usepackage{amsmath}   
\usepackage{amssymb}   
\usepackage{amsthm}    
\usepackage{latexsym}  

\usepackage{enumitem}  
\usepackage{graphicx}  
\usepackage{tikz}      
\usepackage{xcolor}    
\usepackage{hyperref}  
\usepackage{comment}   

\usepackage[style=alphabetic]{biblatex}  
\newtheorem{theorem}{Theorem}[section]
\newtheorem{lemma}[theorem]{Lemma}
\newtheorem{remark}[theorem]{Remark}
\newtheorem{prop}[theorem]{Proposition}

\newtheorem{claim}[theorem]{Claim}

\theoremstyle{definition}
\newtheorem{definition}[theorem]{Definition}
\newtheorem{setup}[theorem]{Setup}

\DeclareMathOperator{\supp}{supp}
\DeclareMathOperator{\aff}{aff}

\newcommand{\mc}{\mathcal}                
\newcommand{\N}{\mathbb{N}}               
\newcommand{\R}{\mathbb{R}}               
\newcommand{\de}{\delta}
\renewcommand{\epsilon}{\varepsilon}
\newcommand{\dir}{\operatorname{dir}}
\newcommand{\GL}{\mathrm{GL}}

\begin{document}

\author{Paige Bright}
\address{
Department of Mathematics\\
Massachusetts Institute of Technology\\
Cambridge, MA 02142, USA}
\email{paigeb@mit.edu}

\author{Alexander Ortiz}
\address{
Department of Mathematics\\
Rice University\\
Houston, TX 77005, USA}
\email{ao80@rice.edu}

\author{Dmitrii Zakharov}
\address{
Department of Mathematics\\
Massachusetts Institute of Technology\\
Cambridge, MA 02142, USA}
\email{zakhdm@mit.edu}
 
\keywords{}
\subjclass[2020]{28A75, 28A78}

\date{}

\title[A Continuum Beck-type Theorem for Hyperplanes]{A Continuum Beck-type Theorem for Hyperplanes}

\begin{abstract}
We prove a sharp continuum Beck-type theorem for hyperplanes.
Our work is inspired by foundational work of Beck on the discrete problem, as well as refinements due to Do and Lund. The inductive proof uses recent breakthrough results in projection theory by Orponen--Shmerkin--Wang and Ren, who proved continuum Beck-type theorems for lines in $\R^2$ and $\R^n$.
\end{abstract}

\maketitle

\tableofcontents

\section{Introduction}

In 1983, J\'{o}zsef Beck \cite{beck1983lattice} studied the following discrete problem about points and lines in Euclidean space. Let $\mathcal P^1(X)$ denote the set of affine lines spanned by at least two points of $X$. Given a finite set $X\subset \R^n$ with $|X| = N$, study the size of the set $\mathcal P^{1}(X)$. Beck proved that a dichotomy holds: either about $N$ points lie on a common line, or there are about $N^2$ distinct lines in $\mathcal P^1(X)$.

Recently, Orponen, Shmerkin, and Wang \cite{orponen2022kaufman} proved a continuum analog of Beck's theorem for lines in the plane, which was soon after extended to lines in higher dimensions by Ren \cite{ren2023discretized}. Ren showed that given $X\subset \R^n$ Borel, either 
\begin{itemize}
    \item[(1)] There exists an $m$-dimensional affine subspace $F\subset\R^n$ such that $\dim (X\setminus F) < \dim X$, or 
    \item[(2)] The set $\mathcal P^1(X)$, a subset of the affine Grassmannian of lines in $\R^n$, satisfies 
    \[
    \dim \mathcal P^1(X) \geq 2\min\{\dim X, m\}.
    \]
\end{itemize}
\noindent Here and throughout the paper, $\dim A$ refers to the Hausdorff dimension of $A$. We refer the reader to Mattila's book \cite{mattila1999} for standard facts about Hausdorff dimension and properties of the affine Grassmannian.

In addition to studying the set of lines that contain at least two points of a given finite set $X \subset \mathbb{R}^n$, Beck also studied hyperplanes of $\R^n$ spanned by $n$ affinely independent points of $X$. For $1\leq k \leq n-1$, let $\mathcal P^k(X)$ denote the set of $k$-planes spanned by $k+1$ affinely independent points of $X$. In his original paper \cite{beck1983lattice}, Beck showed that for any finite set $X\subset \R^n$, either there exists a hyperplane that contains about $|X|$-many points of $X$, or $|\mathcal P^{n-1}(X)|\geq c|X|^n$, for some dimensional constant $c$. The assumption that no hyperplane can contains a large fraction of the points of $X$ is a \emph{non-concentration} hypothesis about the set $X$.

In 2016, Thao Do \cite{do2018extending} and Ben Lund \cite{lund2016essential} refined Beck's theorem for hyperplanes by sharpening the non-concentration hypothesis of $X$. Do showed that for any $\varepsilon>0$ there is some $\gamma_\varepsilon >0$ such that for finite $X\subset \R^n$, either
\begin{itemize}
    \item[(1)] There exists a collection of affine subspaces $F_1,\dots, F_r\subset\R^n$ whose union contains at least $(1-\varepsilon) |X|$ points of $X$ and $\sum_{i=1}^r \dim F_i \leq n-1$, or 
    \item[(2)] $|\mathcal P^{n-1}(X)| \geq \gamma_\varepsilon |X|^n.$
\end{itemize}

In this paper, we establish a continuum analog of Beck's theorem for hyperplanes in $\R^n$. Either a Borel set $X$ is \emph{non-concentrated}, or else the set of hyperplanes spanned by $X$ has maximal dimension.

\begin{definition}
    Let $n \ge 2$, and let $X\subset \R^n$ be Borel. We say that $X$ is \emph{non-concentrated}, or simply $\mathbf{NC}$, if for any $r \ge 1$ and any collection of affine subspaces $F_1,\dots, F_r\subset\R^n$ such that $\sum_{i=1}^r \dim F_i \leq n-1$, we have 
    \[
    \dim \Big(X\setminus \bigcup_{i=1}^r F_i\Big) = \dim X.
    \]
\end{definition}

Heuristically, a Borel set that is $\mathbf{NC}$ corresponds to a finite set disobeying (1) in the dichotomy of Do's Beck-type theorem. Our main result is the following: 

\begin{theorem}\label{thm:main}
Let $X\subset \R^n$ be Borel and $\mathbf{NC}$. Then, 
\[
\dim \mathcal P^{n-1} (X) \geq n\min\{\dim X, 1\}.
\]
\end{theorem}
The lower bound on the dimension of $\mathcal P^{n-1}(X)$ is sharp, and the \textbf{NC} assumption is sharp in the following sense. If $X$ is contained in a union of pairwise disjoint flats $\bigcup_{i=1}^r F_i$ where $\sum \dim F_i \le n-1$, $\dim X >0$, and  $\dim X\cap F_i = \dim X$ for all $i$, then
\[
\dim \mathcal P^{n-1}(X) \le (n-2)\min\{\dim X, 1\}.
\]
To see this, let $V_{x_1,\dots,x_n}$ be a hyperplane spanned by $x_1,\dots,x_n\in X$, and consider
\[
m_i = \#\{x_j : x_j \in F_i\}.
\]
Flats $F_i$ are pairwise disjoint,  so $\sum m_i = n$. Since $\sum \dim F_i \le n-1$, there exists a flat $F_i$ such that $m_i > \dim F_i$. Let $F$ be a flat of codimension $q$, and denote by $\mathcal P_F^{n-1}(X)$ the set of $(n-1)$-planes of $X$ spanned by $X$ containing $F$. It follows that
\[
\mathcal P^{n-1}(X) \subset\bigcup_i \mathcal P^{n-1}_{F_i}(X).
\]
Therefore,
\[
\dim\mathcal P^{n-1}(X) \le \max_{i}(n-(\dim F_i+1)) \le n-2.
\]
Consider the surjective function
\[
\phi:(X\setminus F)^{q-1} \to \bigcup_{j\ge 1}\mathcal P_F^{n-j}(X)
\]
defined by $\phi(x_1,\dots,x_{q-1}) = \langle F,x_1,\dots,x_{q-1}\rangle$. The function $\phi$ is well defined and locally Lipschitz. Therefore, $\dim \mathcal P_F^{n-1}(X) \le (q-1)\dim X$, and $$\mathcal P^{n-1}(X) \le (n-2)\dim X.$$

\subsection{Connections to projection theory}

Given $x\in \R^n$, let $\pi_x : \R^n \setminus \{x\} \to \mathbb{S}^{n-1}$ be the \textit{radial projection onto $x$} given by 
\[
\pi_x(y) = \frac{y-x}{|y-x|}.
\]
In \cite{orponen2022kaufman}, Orponen, Shmerkin, and Wang proved the following radial projection theorem:
\begin{theorem}[Theorem 1.1, \cite{orponen2022kaufman}]\label{thm:osw}
    Let $X\subset\R^2$ be a (non-empty) Borel set which is not contained on any line. Then, for every Borel set $Y\subset \R^2$,
    \[
    \sup_{x\in X}\dim\pi_x(Y\setminus\{x\})\ge \min\{\dim X,\dim Y, 1\}.
    \]
\end{theorem}

Orponen, Shmerkin, and Wang used Theorem \ref{thm:osw} to obtain the sharp continuum Beck-type theorem for lines in the plane. Orponen and Shmerkin \cite{orponen2023ABC} used Theorem \ref{thm:osw} to prove the conjectured Furstenberg set estimates in the plane for AD regular sets. Later, Ren and Wang \cite{ren2023furstenberg} proved the full Furstenberg set conjecture, using the AD regular case in an important way.

In \cite{ren2023discretized}, Ren generalized the Orponen--Shmerkin--Wang's radial projection theorem  to $\R^n$:
\begin{theorem}[Theorem 1.1, \cite{ren2023discretized}]\label{thm:ren}
    Let $X\subset \R^n$ be a (non-empty) Borel set which is not contained in any $m$-plane. Then, for every Borel set $Y\subset \R^n$, 
    \[
    \sup_{x\in X} \dim \pi_x(Y \setminus \{x\}) \geq \min\{\dim X, \dim Y, m\}
    \]
\end{theorem}

Theorem \ref{thm:ren} was used to improve the best upper bounds for the Falconer distance set problem in $\R^n$ for $n\geq 3$ \cite{DuOuRenZhang}.

In our paper, we use discretized versions of Theorems \ref{thm:osw} and \ref{thm:ren} in an inductive scheme to prove a more quantitative version of Theorem \ref{thm:main} (see Theorem \ref{thm:general-case}) which we use to deduce Theorem \ref{thm:main}. The structure of our proof, and the construction of \emph{thin-planes graphs} is geometrically motivated by the notion of \emph{thin-tubes} introduced by Shmerkin--Wang \cite{ShmerkinWang} and developed by Orponen--Shmerkin--Wang \cite{orponen2022kaufman} and Ren \cite{ren2023discretized}.

\subsection{Overview of the proof}

Let $X\subset \R^n$ be Borel and \textbf{NC}, and let  $\mu$ be a $s$-Frostman measure supported on $X$ with $0 < s < \dim X$.

Our approach to this problem is partly inspired by the work of Do \cite{do2018extending}, who refined the geometric non-concentration condition in Beck's dichotomy. Our strategy is to begin by identifying \emph{irreducible} pieces of the measure $\mu$. 

\begin{definition}
    We say that a measure $\mu$ supported on an affine subspace $V\subset \R^n$ is \emph{irreducible in $V$} if $\mu(V') = 0$ for any proper affine subspace $V'\subset V$.
\end{definition}

\textbf{Case 1.} If $\mu$ is irreducible on $\R^n$, we will have a direct argument that $\mathcal P^{n-1}(X)$ has the correct dimension using Ren's radial projection theorem \cite{ren2023discretized}. In particular, we prove the following.
\begin{lemma}\label{lem:introirreducible}
    Let $X \subset \R^n$ be Borel with $\dim X>0$, and let $0 < s < \dim X$. Given an irreducible (in $\R^n$) $s$-Frostman measure $\mu$ supported on $X$, it follows that
    \[
    \dim \mathcal P^{n-1}(X) \geq n \min\{s,1\}.
    \]
\end{lemma}
\noindent In particular, if there exists such a $\mu$ for all $0 < s < \dim X$, then sending $s \nearrow \dim X$ proves our main result. We refer to this as the \emph{irreducible case} of our continuum Beck-type theorem.

\textbf{Case 2.} If $V$ is a \emph{proper} affine subspace in $\R^n$ such that $\mu|_V$ is irreducible on $V$, it follows by Lemma \ref{lem:introirreducible} that $X\cap V$ spans a set of hyperplanes in $V$ of the correct dimension. (If $\dim V = 1$, we mean that $X\cap V$ has the appropriate dimension.) 

Since $X$ is \textbf{NC}, we know that $\dim (X\setminus V) = \dim X$. Thus, we can iterate through the two cases described above for the set $X\setminus V$. This gives us an algorithmic method for finding a collection of affine subspaces supporting irreducible measures, which we record in the following lemma.

\begin{lemma}\label{lem:irreduciblemeasures}
    Let $X \subset \R^n$ be Borel and $\mathbf{NC}$ with $\dim X > 0$. Then for any $0<s< \dim X$, there exist a collection of affine subspaces $V_1, \dotsc, V_m \subset \R^n$ and $s$-Frostman measures $\mu_i$ supported on $X \cap V_i$, $i = 1,\dots,m$, such that
    \begin{itemize}
        \item[(i)] Each $\mu_i$ is irreducible in $V_i$ and $\mu_i(X) > 0$.
        \item[(ii)] The supports $\supp\mu_i$ are pairwise disjoint.
        \item[(iii)] For any affine subspaces $F_1, \dotsc, F_r$ such that $\bigcup_{i=1}^m V_i \subset \bigcup_{j=1}^r F_j$ we have $\sum_{j=1}^r \dim F_j \ge n$.
    \end{itemize}
\end{lemma}

Let $n_i = \dim V_i$. Notice then that applying Lemma \ref{lem:introirreducible} to each $V_i$, we see that $\dim \mathcal P^{n_i-1}(X\cap V_i) \geq n_i\min\{s,1\}$ for any $0 <s < \dim X$. Again, recall here that $\mathcal P^0(X\cap V_i)= X\cap V_i$ in the case where $n_i = 1$. From here, we will ``patch together'' codimension 1-planes from each $V_i$ to build hyperplanes in $\R^n$. 

The guiding idea is to choose $n_i$ affinely independent points from each $X\cap V_i$ so that their union spans a hyperplane in $\R^n$. In practice, not every such selection works: the chosen points may span a space of dimension smaller than $n-1$, or they may already fill all of $\R^n$ if $\sum_{i=1}^m n_i \ge n.$ To avoid these issues, we will take care to ensure that our collection of measures $\mu_1,\dots,\mu_r$ is in what we call \emph{$c$-stable position}, so that, for example, the angle between any two distinct flats $V_i,V_j$ is bounded below by $\theta(c)>0$, and (say) the dimension of intersection $\dim(V_i\cap V_j)$ is what one would expect under generic conditions. The use of $c$-stable position will facilitate the quantitative argument of Section \ref{sec:minimal-case}.

\subsection{Paper outline}

We begin by making use of the \textbf{NC} condition on $X$, and in particular prove Lemma \ref{lem:irreduciblemeasures} in Section \ref{sec:irreducibledecomp}, which allows us to find a collection of affine subspaces $V_i$ and corresponding irreducible $s$-Frostman measures $\mu_i$ for any $0 < s < \dim X$. The setup is similar to that in the radial projection work of Orponen--Shmerkin--Wang \cite{orponen2022kaufman} and Ren \cite{ren2023discretized}, which we unpack in Section \ref{sec:thinplanesprelim}. We give background on the notion of \emph{thin tubes}, and describe a generalization we refer to as \emph{thin $k$-planes}. We prove Lemma \ref{lem:introirreducible} in Section \ref{sec:qualititative-irreducible-beck} by showing that irreducible measures span thin hyperplanes. 

Next, we ``patch together'' the measures on each $V_i$ to form hyperplanes in $\R^n$. Some computations will be justified through the use of the dimension-sum formula of linear algebra: for linear subspaces $V,W\subset \R^n$,
\[
\dim(V+W) = \dim V + \dim W - \dim V\cap W.
\]
The assumption that our collection of measures $\{\mu_i\}$ are in $c$-stable position in $\{V_i\}$ will guarantee that various ``joins'' $V+W$ and ``meets'' $V\cap W$ have the expected dimension. The reduction to $c$-stable position is carried out in Section \ref{sec:stable-position}.

Furthermore, we will be able to assume that our collection of flats is \textit{minimal} in the following sense.

\begin{definition}\label{def:minimal-flats}
Let $F_1,\dots, F_k$ be a collection of affine flats in $\R^n$. Furthermore, let $[k] = \{1,\dots, k\}$, and given $J \subset [k]$ let 
\[
F_J := \mathrm{aff}(F_J : j \in J),
\]
the smallest dimensional affine flat containing $F_j$ for all $j \in J$. We say that a collection of affine flats $F_1, \ldots, F_k$ is \emph{minimal} if $\dim F_{[k]} = n \le \sum_{j=1}^k \dim F_j$ and $\dim F_J \ge \sum_{j\in J} \dim F_j$ for any proper $J \subsetneq [k]$.    
\end{definition}

The reduction to flats that are minimal is done in Section \ref{sec:reduction-to-minimal}. The main work that remains is dealing with this minimal case, which is the entirety of Section \ref{sec:minimal-case}.

\subsection{Notation}

We denote by $\aff(x_0,\dots,x_k)$ or $\langle x_0,\dots,x_k\rangle$ the (at most) $k$-plane spanned by points $x_0,\dots,x_k\in \R^n$.

For $k \le n-1$, we denote by
\[
\mathcal S_k = \{(x_0,\dots,x_k)\in (\R^n)^{k+1}:x_0,\dots,x_k\ \text{belong to a common $(k-1)$-plane}\}.
\]
By $\mathcal A(V,k)$ we denote the set of affine $k$-dimensional subspaces of $V$, and $\mathcal A(n,k) = \mathcal A(\R^n,k)$, with similar notation $\mathcal G(V,k), \mathcal G(n,k)$ for linear $k$-dimensional subspaces.

\bigskip
\begin{sloppypar}
\noindent {\bf Acknowledgments.} The first author would like to thank Larry Guth, Pablo Shmerkin, and Josh Zahl for supportive discussions during the duration of this project. The first author was supported by the MathWorks fellowship at MIT.
\end{sloppypar}

\section{Decomposing into irreducible pieces}
\label{sec:irreducibledecomp}

In this section, we prove Lemma \ref{lem:irreduciblemeasures}, which we restate here.

\begin{lemma}\label{lem:decomposition}
    Let $X \subset \R^n$ be Borel and $\mathbf{NC}$ with $\dim X > 0$. Then for any $0<s< \dim X$, there exist a collection of affine subspaces $V_1, \dotsc, V_m \subset \R^n$ and $s$-Frostman measures $\mu_i$ supported on $X \cap V_i$, $i = 1,\dots,m$, such that
    \begin{itemize}
        \item[(i)] Each $\mu_i$ is irreducible in $V_i$ and $\mu_i(X) > 0$.
        \item[(ii)] The supports $\supp\mu_i$ are pairwise disjoint.
        \item[(iii)] For any affine subspaces $F_1, \dotsc, F_r$ such that $\bigcup_{i=1}^m V_i \subset \bigcup_{j=1}^r F_j$ we have $\sum_{j=1}^r \dim F_j \ge n$.
    \end{itemize}
\end{lemma}

\noindent Recall that we say that a Borel set $X \subset \R^n$ is \textbf{NC} if 
$$
\dim \Big(X \setminus \bigcup_{i=1}^j F_i\Big) = \dim X,
$$
for any affine flats $F_1, \dotsc, F_j$ with $\sum_{i=1}^j \dim F_i \le n-1$. Additionally, recall that we say that a measure $\mu$ supported on an affine subspace $V \subset \R^n$ is {\em irreducible in $V$} if $\mu(V') = 0$ for any proper affine subspace $V' \subset V$. 

\begin{proof}[Proof of Lemma \ref{lem:decomposition}]
Let $0 < s < \dim X$. We say that a flat $V$ is \emph{large} if there exists an irreducible (in $V$) $s$-Frostman measure $\mu$ on $V \cap X$ with $\mu(X) > 0$.

For a collection of flats $\mc V = \{V_1, \dotsc, V_m\}$ and a partition $\{I_1, \dots ,I_j\}$ of $[m]$, define
$$
c(\mc V; I_1, \dotsc, I_j) = \sum_{i=1}^j \dim \aff(V_t: t\in I_i),
$$
where $\aff(E_i:i\in I)$ means the affine span of a collection of flats, i.e., the smallest affine subspace containing all the sets $E_i$ for $i \in I$. Let us define the \emph{cost} of a set of $m$ flats to be the minimal dimension sum as we vary partitions of $[m]$:
\begin{equation}\label{eqn:cost}
c(\mc V) = \min_{\{I_1, \dotsc, I_j\}} c(\mc V; I_1, \dotsc, I_j),
\end{equation}
where the minimum is taken over all partitions of $[m]$. If a partition $\{I_1,\dots,I_j\}$ of $[m]$ achieves the cost of $\mathcal V$, then we say that this partition is \emph{minimizing}.

We will construct a family of large flats $\mc V$ and corresponding measures $\mu_i$ inductively, using the cost to ensure that our inductive construction eventually stops. Let $\mc V_0 = \emptyset$ (in this case we have $c(\mc V_0) = 0$). Given a collection of $m$-many large flats $\mc V_m = \{V_1,\dots, V_m\}$ and corresponding measures $\mu_1,\dots,\mu_m$ (in case $m = 0$, we do not define $\mu_0$), if $c(\mc V_m) \ge n$ then stop. Otherwise, we define $\mc V_{m+1}$ as follows. 

By assumption, we have $c(\mc V_m) = c \leq n-1$ and there exists some partition $\{I_1, \dotsc, I_j\}$ of $[m]$ such that $c(\mc V_m;I_1, \dotsc, I_j) = c$.  Define $F_i = \aff(V_t: t\in I_i)$ and 
note that given $c\leq n-1$ and $X$ is $\mathbf{NC}$, we have
$$
\dim \left(X \setminus \bigcup_{i=1}^j F_i\right) = \dim X.
$$
By Frostman's lemma, there exists an $s$-Frostman probability measure $\mu$ with
\begin{equation}\label{eqn:disjointsupport}
\supp\mu \subset X \setminus \bigcup_{i=1}^jF_i.
\end{equation}
By shrinking the support of $\mu$ if necessary, we can ensure that $\supp\mu$ is separated from $\bigcup_{i=1}^jF_i$. Let $V_{m+1} \subset \R^n$ be a subspace of minimal dimension such that $\mu(V_{m+1}) > 0$. Note that since $\mu(\mathbb R^n) = 1$, such a minimal subspace necessarily exists. Then it is clear that $V_{m+1}$ is large and $\mu|_{V_{m+1}}:= \mu_{m+1}$ is the corresponding irreducible Frostman measure on $V_{m+1}$. Note that $\mu_{m+1}$ has pairwise disjoint support from any previously constructed measures $\mu_1,\dots, \mu_m$ obtained in this way by \eqref{eqn:disjointsupport}. Let $\mc V_{m+1} = \mc V_m \cup \{V_{m+1}\}$ and proceed inductively.

We claim that this process eventually stops, i.e., eventually $c(\mc V_m) \geq n$. Note that $c(\mc V_{m+1}) \ge c(\mc V_{m})$ by construction. Indeed, given a partition $\{I'_1, \dotsc, I'_{j}\}$ of $[m+1]$ we can define a partition of $[m]$ by setting $I_i = I'_i \setminus \{m+1\}$ and noting that 
$$
c(\mc V_m; I_1, \dotsc, I_j) \le c(\mc V_{m+1}; I'_1, \dotsc, I'_j).
$$
Now we need to rule out that we ``get stuck'' on some fixed value $c(\mc V_m) = c \leq n-1$. To do so, define
$$
N(\mc V_m) = \#\{\text{partitions $\pi$ of $[m]$}: c(\mathcal V_m;\pi) = c(\mc V_m)\},
$$
i.e., the number of partitions of $[m]$ achieving the cost of $\mathcal V_m$. Note that this number is finite, and furthermore, by definition of $c(\mc{V}_m)$, $N(\mc{V}_m) \geq 1$ for all $m \ge 1$.

Suppose for the sake of contradiction that we do ``get stuck'' on some value of $c\leq n-1$, i.e., the algorithmic process has given us a countably infinite set of collections $\{\mc{V}_i\}_{i\geq m}$ such that $c(\mc{V}_i) = c$ for all $i\geq m$. We then show that $N(\mc{V}_i)$ is a strictly decreasing sequence, i.e.,
\[
N(\mc{V}_{i+1}) < N(\mc{V}_i),\quad \forall i \geq m.
\]
This will give a contradiction, as $N(\mc{V}_m)$ is finite, and for all $i\geq 1$, $N(\mc{V}_i) \geq 1$.

Fix an arbitrary $i\geq m$. Firstly, notice that given an arbitrary partition $\{I_1,\dots, I_j\}$ of $[i]$, we may consider the partition of $[i+1]$ given by $\{I_1, \dots , I_j, \{i+1\}\}$. For such a partition, we have
\begin{align*}
c(\mc{V}_{i+1}; I_1,\dots, I_j, \{i+1\}) &= c(\mc{V}_i;I_1,\dots, I_j) + \dim  V_{i+1}\\
&\ge c(\mathcal V_i) + 1 > c.
\end{align*}
In other words, any partition of $[i+1]$ which contains the set $\{i+1\}$ does \emph{not} minimize the cost of $\mathcal V_{i+1}$, and hence does not contribute to $N(\mc{V}_{i+1})$.

Furthermore, suppose that $\{I_1,\dots,I_j\}$ is a partition of $[i]$ such that 
\[
c(\mc{V}_i;I_1,\dots, I_j) > c.
\]
Suppose that we add the element $i+1$ to one of the sets $I_\ell$ for some $1\leq \ell \leq j$. We will call this an \emph{extension} of the partition $\{I_1,\dots, I_j\}$. Then, it follows that 
\[
c(\mc{V}_{i+1}; I_1,\dots, I_\ell \cup \{i+1\},\dots , I_j) \geq c(\mathcal V_{i}; I_1,\dots, I_\ell,\dots,I_j)>c.
\]
Hence, any extension of a non-minimizing partition of $\mc{V}_i$ is non-minimizing for the collection $\mc{V}_{i+1}$, so it cannot contribute to $N(\mathcal V_{i+1})$ either.

Therefore, the only partitions which contribute to $N(\mathcal V_{i+1})$ are those which are extensions of partitions which achieve the cost of $\mathcal V_i$.
Thus, the proof that $N(\mc{V}_{i+1}) < N(\mc{V}_i)$ reduces to the following claim.

\begin{claim}\label{claim:extensions}
\noindent
\begin{enumerate}[label=\alph*), leftmargin=2em]
\item For any partition $\{I_1, \dotsc, I_j\}$ of $[i]$ such that $c(\mc V_i; I_1, \dotsc, I_j) = c$, there exists at most one extension of this partition to a partition $\{I_1',\dots, I_j'\}$ of $[i+1]$ with $c(\mathcal V_{i+1}; I'_1, \dotsc, I'_j) = c$. Here, again, extension means that we simply add the element $i+1$ to one of $I_\ell$ for $1 \leq \ell \leq j$.

\item Furthermore, there exists at least one minimizing partition $\{I_1,\dots, I_j\}$ of $[i]$ such that there are no minimizing extensions of this partition.
\end{enumerate}
\end{claim}

\begin{proof}[Proof of Claim \ref{claim:extensions}]
Note that the trivial partition of $[i]$, namely $\{[i]\}$ consisting of only one set, clearly has at most one (potentially non-minimizing) extension. Therefore, we only need to consider partitions $\{I_1,\dots,I_j\}$ of $[i]$ where $j\ge 2$.

We begin with proving part a). Consider a partition $\{I_1,\dots,I_j\}$ of $[i]$, where $j \ge 2$, and $c(\mathcal V_i; I_1,\dots,I_j) = c$. Suppose for the sake of contradiction that there exists a pair of indices $\ell \ne \ell'$ such that both extensions
$$
I_1, \dotsc, I_\ell\cup\{i+1\},\dotsc, I_j \text{ and }  I_1, \dotsc, I_{\ell'}\cup\{i+1\},\dotsc, I_j
$$
are minimizing.
This is equivalent to saying that
$$
V_{i+1} \subset \aff(V_t; t\in I_\ell)\text{ and }V_{i+1} \subset \aff(V_t; t\in I_{\ell'}).
$$
Since $s > 0$, and $V_{i+1}$ supports an $s$-Frostman measure, we have $\dim V_{i+1} \geq 1$. So it follows that
$$
\dim \aff(V_{t}; t \in  I_\ell \cup I_{\ell'}) \leq \dim \aff(V_{t}; t \in  I_\ell) + \dim \aff(V_{t}; t \in  I_{\ell'}) - 1.
$$
Since $\{I_1, \dots, I_j\}$ is a partition of $[i]$, so is
$$
\{I_\ell\cup I_{\ell'}, I_1,\dots,\widehat{I_\ell},\dots,\widehat{I_{\ell'}},\dots, I_j\},
$$
where the notation $\widehat{\bullet}$ means that the element $\bullet$ is removed from the list. Therefore,
$$
c(\mc V_i;  I_\ell\cup I_{\ell'}, I_1, \dotsc, \widehat{I_\ell}, \dotsc,\widehat{ I_{\ell'}},\dots, I_j) < c(\mc V_i; I_1,\dots, I_\ell,\dots, I_{\ell'},\dots, I_j) = c,
$$
which contradicts the minimality of $c = c(\mathcal V_i)$. This shows that for every minimizing partition of $\mc{V}_i$, there exists at most one minimizing extension of said partition.

Now we will show part b) of Claim \ref{claim:extensions}: that there exists a minimizing partition $\{I_1,\dots, I_j\}$ of $[i]$ with no minimizing extensions. Indeed, let $\{I_1,\dots, I_j\}$ be the partition of $[i]$ which was used to define $V_{i+1}$.

Recall that we found $V_{i+1}$ in such a way that
\[
V_{i+1}\not\subset F_\ell = \aff(V_t; t\in I_\ell), \text{ for any } 1\leq \ell \leq j.
\]
 Therefore, it must be the case that for all $1\leq \ell \leq j$, 
 \[
 \dim \aff(V_t; t\in I_\ell \cup \{i+1\}) > \dim \aff(V_t; t\in I_\ell).
 \]
That is, any extension $\{I_1',\dots, I_j'\}$ of the partition $\{I_1,\dots, I_j\}$ must satisfy 
\[
c(\mc{V}_{i+1}; I_1',\dots, I_j') > c(\mc{V}_i; I_1,\dots, I_j) = c.
\]
This concludes the proof of part b), and of Claim \ref{claim:extensions}.
\end{proof}

Claim \ref{claim:extensions} contradicts the assumption that there was a countably infinite sequence $\{\mc{V}_i\}_{i=m}^{\infty}$ such that $c(\mc{V}_i) = c$ for all $i\geq m$. It follows that the value of $c(\mathcal V_i)$ must eventually increase. Additionally, since the algorithm terminates if $c\geq n$, it follows that the algorithm terminates after running it a finite number of times. 

We let $\mathcal V_m = \{V_1, \dotsc, V_m\}$ be a final collection of subspaces with corresponding $s$-Frostman measures $\{\mu_1,\dots, \mu_m\}$ obtained at the termination of the algorithm.

By construction, the measures $\mu_i$ are irreducible in $V_i$, and they have pairwise disjoint supports, which is parts (i) and (ii) of Lemma \ref{lem:decomposition}. We now verify that property (iii) holds for $\mathcal V_m = \{V_1,\dots,V_m\}$. To see this, suppose that $\bigcup_{i=1}^m V_i\subset\bigcup_{j=1}^r F_j$. The reader can verify that for each $i$, there is some $j(i)$ such that $V_i \subset F_{j(i)}$. Define a function $j\colon [m]\to [r]$ by sending each $i\in[m]$ to a corresponding $j(i)\in[r]$ in this way, and let $S = \mathrm{image}(j)$. By definition, $\{j^{-1}(x): x\in S\}$ is a partition of $[m]$. Furthermore, by the construction of $\mathcal V_m$, we must have
\[
n \le c(\mathcal V;j^{-1}(x),x\in S) := \sum_{x\in S}\dim \aff(V_i;i\in j^{-1}(x)).
\]
Since each $i\in j^{-1}(x)$ satisfies $V_i\subset F_x$, and $F_x$ is closed under affine combinations, we have 
$$
\sum_{x\in S}\dim \aff(V_i;i\in j^{-1}(x))\le \sum_{x\in S}\dim F_x,
$$
and hence $n \le \sum_{j=1}^r \dim F_j$.
\end{proof}

\section{\texorpdfstring{Thin $k$-planes}{Thin k-planes}} \label{sec:thinplanesprelim}

In the previous section we decomposed an \textbf{NC} set or measure into a collection of affine subspaces supporting irreducible $s$-Frostman measures. In this section, we will describe what it means for a collection of measures to span $s$-thin hyperplanes.

The works of Orponen--Shmerkin--Wang \cite{orponen2022kaufman} and Ren \cite{ren2023discretized} obtain continuum versions of Beck's theorem for lines via radial projection theorems, which in turn are proven using the notion of \emph{thin tubes}. We recall their definition of thin tubes here. If $X_0,X_1,\dots, X_k\subset \R^n$, and $G\subset \prod_i X_i$, for all $x_j \in X_j$ we define the $x_j$-section of $G$:
\[
G|_{x_j} = \left\{(x_0,\dots, \widehat{x}_j,\dots x_k) \in \prod_{i\neq j} X_i : (x_0,\dots,x_j,\dots, x_k) \in G\right\}.
\]
If $G$ is Borel, then every section is measurable too.
\begin{definition}[Thin tubes]\label{defn:thin-tubes}
    Let $K,\sigma\geq 0$ and $c\in (0,1]$. Let $\mu_0,\mu_1$ be probability measures on $\R^n$ with $\supp(\mu_i) = X_i$ for $i=0,1$. We say that $(\mu_0,\mu_1)$ has $(\sigma,K,c)$-thin tubes if there exists a Borel set $G\subset X_0 \times X_1$ with $(\mu_0\times \mu_1) (G) \geq c$ with the following property. If $x_0 \in X_0$, then 
    \[
    \mu_1(T\cap G|_{x_0}) \leq K \cdot r^\sigma  \quad \text{for all $r>0$ and all $r$-tubes $T$ containing $x.$}
    \]
    We also say $(\mu_0,\mu_1)$ has $\sigma$-thin tubes if $(\mu_0,\mu_1)$ has $(\sigma,K,c)$-thin tubes for some $K\geq 0$ and $c\in (0,1].$
\end{definition}

One can readily check that if $(\mu_0,\mu_1)$ and $(\mu_1,\mu_0)$ have $\sigma$-thin tubes and $\supp \mu_i \subset X$ for $i=1,2$ it follows that $\dim \mathcal L(X) \ge 2\min\{\sigma,1\}$. Our definition of $\sigma$-thin $k$-planes will be formulated to ensure that the set of $k$-planes spanned by sets $X_0,\dots,X_k$ in ``good position'' is at least $(k+1)\min\{\sigma,1\}$-dimensional. In particular, we will deduce the sharp Beck-type theorem for hyperplanes.

Define the partial map $\psi_k\colon (\R^n)^{k+1} \to {\mc A}(n,k)$ which sends a $(k+1)$-tuple of points $(x_0, \dotsc, x_k)$ to the affine $k$-plane $\aff(x_0, \dotsc, x_k)$. This map is well-defined outside of
\[
\mathcal S_k = \{(x_0,\dots,x_{k})\in(\R^n)^{k+1}:x_0,\dots,x_k\ \text{belong to a common $(k-1)$-plane}\}.
\]

\begin{definition}[Measures in $C$-good position]
    We say that measures $\mu_0, \dotsc, \mu_k$ on $\R^n$ are in $C$-\emph{good position} if their supports $\supp\mu_i = X_i$ are contained in $B^n(0,1)$, and $X_0 \times \dots \times X_k$ is $1/C$-separated from $\mc S_k$. If the constant $C$ is not important, we will just refer to measures in \emph{good position}.
\end{definition}

Notice that given $\mu_0,\dotsc, \mu_k$ on $\R^n$ in good position, every $(x_0,\dotsc, x_k) \in \prod X_i$ uniquely determines a $k$-plane containing $x_0,\dots, x_k$. We will denote this $k$-plane $V_{x_0,\dots, x_k}$; we denote the $\de$-neighborhood of this $k$-plane (as a subset of $\R^n$) as $V_{x_0,\dots, x_k}(\de).$ We can now define thin $k$-planes.

\begin{definition}[Thin $k$-planes]
    Let $K \ge 1, \sigma > 0, c \in (0,1)$, and $1 \le k \le n-1$. Let $\mu_0, \dotsc, \mu_k$ be probability measures on $\R^n$ in good position.

    We say that $(\mu_0, \dotsc, \mu_k)$  \emph{span $(\sigma, K, c)$-thin $k$-planes}, or simply have \emph{$\sigma$-thin $k$-planes}, if there exists a Borel set $G \subset \supp\mu_0 \times \dots \times \supp\mu_k$ such that
    \begin{itemize}
        \item[(i)] $(\mu_0\times \dots \times \mu_k)(G) \ge c $,
        \item[(ii)] for any $(x_0, \dotsc, x_k) \in G$ and any $\delta > 0$ and $j = 0, \dotsc, k$, we have 
        $$
        \mu_j(V_{x_0, \dotsc, x_k}(\delta)) \le K \delta^\sigma.
        $$
    \end{itemize}
    We call any such $G$ a \emph{thin $k$-plane graph} of the collection $(\mu_0, \dotsc, \mu_k)$ (and a \emph{$(\sigma,K, c)$-thin $k$-plane graph} if we want to emphasize the values of parameters). We also say that such a $G$ \emph{witnesses} $(\sigma,K,c)$-thin $k$-planes for the measures $\mu_0,\dots,\mu_k$.

    We also extend this definition to the case of non-probability measures $\mu_i$ by saying that $(\mu_0, \ldots, \mu_k)$ span $(\sigma, K, c)$-thin $k$-planes if the renormalized measures $\tilde \mu_i = \frac{1}{\mu(\R^n)} \mu_i$ do.
\end{definition}

Now that we have defined $\sigma$-thin $k$-planes, we can show that if a set $X$ supports $\mu_0, \ldots, \mu_k$ having $\sigma$-thin $k$-planes, then the dimension of the set of $k$-planes spanned by $X$ is at least $(k+1)\sigma$.

\begin{lemma}\label{lem:thin-planes-implies-pushforward-frostman}
    Let $\mu_0,\dotsc, \mu_k$ be probability measures on $\R^n$ in good position with $(\sigma,K,c)$-thin $k$-planes. Let $G$ be a thin $k$-plane graph of the collection $(\mu_0,\dots, \mu_k).$ Then, the pushforward measure $\psi_k(\mu_0\times \cdots \times \mu_k|_G)$ is a nonzero compactly supported $\sigma'$-Frostman measure in $\mathcal A(n,k)$, where $\sigma' = (k+1) \sigma$. In particular, $\sigma \le n - k$.
\end{lemma}

\begin{proof}
    Let $V\in \mathcal A(n,k)$ be arbitrary and $\rho = \mu_0\times\dotsb\times\mu_k$. Let $B_{\mc A(n,k)}(V, \delta)$ denote the $\delta$-neighborhood of $V$ in the space of $k$-flats. That $\psi_k(\rho|_G)$ is nonzero and compactly supported is immediate from $\rho(G) \ge c>0$ and the compact support of $G$.
    
    We want to show that 
    \begin{equation}\label{eq:thin-frostman-sigma-prime}
    \psi_k(\rho|_G) (B_{\mc A(n, k)}(V, \de)) \le K' \de^{\sigma'}
    \end{equation}
    for some constant $K'$ to be determined. Note that there exists a constant $C=C(n)$ so that for every $V' \in B_{\mc A(n, k)}(V, \de)$ we have $V' \cap B^n(0,1) \subset V(C\delta)$. 
    
    Without loss of generality, we suppose that there exists $(x_0,\dots, x_k) \in G$ so that $V_{x_0,\dots, x_k} \in B_{\mc A(n, k)}(V,\delta)$. If this weren't the case, then 
    \[
    \psi_k (\rho|_G) (B_{\mc A(n, k)}(V,\de)) = 0
    \]
    which immediately implies \eqref{eq:thin-frostman-sigma-prime} holds. Thus, we have $V(C\de)\cap B^n(0,1) \subset V_{x_0,\dots, x_k}(C^2\de)$ for some $(x_0,\dots,x_k)\in G$. Thus, if we have $V_{\tilde x_0, \ldots, \tilde x_k} \in B_{\mc A(n, k)}(V,\de)$ for some $(\tilde x_0, \ldots, \tilde x_k)\in G$ then we get $\tilde x_0, \ldots, \tilde x_k \in V_{x_0,\dots, x_k}(C^2\de)$.
    Hence, 
    \begin{align*}
       \psi_k(\rho|_G) (B_{\mc A(n, k)}(V, \de)) &\leq  \int_{G} \prod_{i=0}^k\mathbf{1}_{V_{x_0,\dots, x_k}(C^2\de)}( \tilde{x}_i) \,d\rho(\tilde x_0,\dots,\tilde x_k) \\
&\le \prod_{i=0}^k \mu_i(V_{x_0,\dots, x_k}(C^2\de)) \\
&\leq \prod_{i=0}^k K (C^2\de)^\sigma \\
&= K^{k+1}C^{2(k+1)\sigma} \de^{(k+1)\sigma}.
\end{align*}
The second-to-last line follows from the fact that $G$ is a $(\sigma,K,c)$-thin plane graph of the collection $(\mu_0,\dots, \mu_k)$. Hence, we have shown that $\psi_k(\mu_0\times \cdots \times \mu_k|_G)$ is a nonzero $(k+1)\sigma$-Frostman with constant $K^{k+1} C^{2(k+1)\sigma}$.  
This concludes the proof of the Lemma.
\end{proof}

By Lemma \ref{lem:thin-planes-implies-pushforward-frostman}, to prove Theorem \ref{thm:main}, it suffices to show that if $X$ is an \textbf{NC} set, then for every $\sigma < s < \dim X$, there are $s$-Frostman measures $\mu_1,\dots,\mu_n$ in good position and supported in $X$ that span $\sigma$-thin hyperplanes. As a first step, we show that if $\mu$ is an irreducible $s$-Frostman measure in $\R^n$, then $\mu$ can be decomposed into measures $\mu_1,\dots,\mu_n$ that span $\sigma$-thin $(n-1)$-planes.

\section{The irreducible case}
\label{sec:irreducible implies saturated}

In this section, we prove the following theorem.

\begin{theorem} \label{thm:qualitative-irreducible-beck}
    If $X\subset \R^n$ is a Borel set supporting an irreducible $s$-Frostman probability measure, then for all $1 \leq k \leq n-1$,
    \[
    \dim \mathcal P^{k}(X) \ge (k+1)\min\{s,n-k\}.
    \]
\end{theorem}

\begin{remark}
By translation and dilation, it suffices to prove Theorem \ref{thm:qualitative-irreducible-beck} in the case that $X\subset B^n(0,1)$.  Irreducibility is a continuum analog of the nonconcentration condition in Beck's discrete hyperplane theorem \cite[Theorem 5.4]{beck1983lattice}.
\end{remark}

We prove Theorem \ref{thm:qualitative-irreducible-beck} by establishing a more quantitative version of it. It relies on a quantitative measurement of the irreducibility of a measure.

\begin{definition}\label{def:discretized-irreducible}
    Let $w, \tau>0$ and $V \subset \R^n$ a subspace. 
    We say that a measure $\mu$ supported on $V \cap B^n(0,1)$ is $(w,\tau)$-irreducible on $V$ if for any proper subspace $H\subset V$ we have $\mu(H(w)) \le \tau \mu(B^n(0,1))$. 
\end{definition}

The following discretized version of Theorem \ref{thm:qualitative-irreducible-beck} is the main technical result of this section.

\begin{theorem} \label{theorem:section4main}
    Let $1\leq k \leq n-1$, $0<s\leq n-1$, $0<\sigma < s$ and $C>0$. For every $\varepsilon>0$, there exist $\tau_0:=\tau_0(\varepsilon, s, \sigma, n, C)$ and $w_0 := w_0(\varepsilon, s, \sigma, n, C)$ such that the following holds for all $\tau \leq \tau_0$ and all $w\leq w_0$. Given $\mu_0,\dots, \mu_k$ that are $(w,\tau)$-irreducible $(C,s)$-Frostman measures in $C$-good position, there exists  $K:= K(\varepsilon, s,\sigma, n, C, w)>0$ such that $(\mu_0,\dots, \mu_k)$ spans $(\sigma, K, 1-\varepsilon)$-thin $k$-planes.
\end{theorem}

\begin{proof}[Proof of Theorem \ref{theorem:section4main}]

We prove Theorem \ref{theorem:section4main} by induction on $k$ and $n$. Ren's discretized radial projection theorem \ref{thm:ren-thin-tubes} will imply Theorem \ref{thm:renprojcorollary}, the base case of our induction when $k = 1$ and $n\geq 2$ is arbitrary.

We now proceed to the induction step, so let us fix $2\leq k\leq n-1$ and suppose that the statement holds for all smaller values of $n$ and all smaller values of $k$ (within the allowable range). Let $X_i = \supp \mu_i$. As a first step, we apply the base case $k=1$ to each pair $\mu_i, \mu_j$ $(i\ne j)$. Since the measures $\mu_0,\dots,\mu_k$ are in $C$-good position, every pair $\mu_i,\mu_j$  is also in $C$-good position. By the base case $k=1$ applied to each pair, there exist Borel sets $G_{ij} \subset X_i \times X_j$  that witness $(\sigma, K, 1-\varepsilon)$-thin $1$-planes for the measures $\mu_i,\mu_j$ for arbitrary $\sigma < s$ for some $K_0 = K_0(\varepsilon, s,\sigma, n, C, w)$. We may assume that $(x_i, x_j) \in G_{ij}$ if and only if $(x_j, x_i) \in G_{ji}$. 
    
    Next, we define 
    \begin{equation}\label{eq:big-projection}
    \widetilde X_i = \{ x_i \in X_i:~ \mu_j(G_{ij}|_{x_i}) \ge 1 - \varepsilon^{1/2}, ~ \text{ for any }j\neq i\}.    
    \end{equation}
    We claim that $\mu_i(\widetilde X_i) \ge 1-   \varepsilon^{1/2}$ holds. Indeed by Fubini,
    $$
    1-\varepsilon \le \mu_i\times \mu_j(G_{ij}) = \int \mu_j(G_{ij}|_{x_i})\, d\mu_i(x_i) \le \mu_i(\widetilde X_i) + \mu_i(X_i \setminus \widetilde X_i) (1-\varepsilon^{1/2})
    $$
    and rearranging gives the bound (using that $\mu_i(X_i\setminus \tilde{X}_i) = 1-\mu_i(\tilde{X}_i)$).

    Consider a fixed hyperplane $H$ with $\mathrm{dist}(H, X_i) > C$ for every $0\le i \le k$. Denote by $\mathrm{dir}(H)$ the linear subspace of $\R^n$ tangent to $H$.  
    Fix $0\leq i \leq k$ and let $x_i \in \widetilde{X}_i$. We define a ``radial'' projection map onto the hyperplane $H$ given by 
    \begin{align*}
    \pi_{x_i} : ~ &\R^n \setminus (\mathrm{dir}(H) +x_i) \to H \\
    &\pi_{x_i}(y) = \langle x_i, y\rangle \cap H.
    \end{align*}
    Consider the measures $\{\pi_{x_i}(\mu_j)\}_{j:j\neq i}$ on $H$, which we identify with $\R^{n-1}$. We claim that these $k-1$ measures satisfy the hypotheses of Theorem \ref{theorem:section4main} in $\R^{n-1}$.

    Up to rescaling by $\sim 1/C$, we have $\pi_{x_i}\mu_j$ supported on a unit ball $B^{n-1}$ in $H$ and $\{\pi_{x_i}\mu_j\}_{j:j\ne i}$ is in $C^2$-good position in $\R^{n-1}$. We observe that for a $(n-2)$-plane $U \subset H$ we have $\pi_{x_i}^{-1}(U(\de)) \subset (\pi_{x_i}^{-1}(U))(A\delta)$ for some  $A$ (depending on how much $H$ and $\supp \mu_i$ are separated). It follows that for $i\neq j$ and every $x_i \in \supp\mu_i$, $\pi_{x_i}\mu_j$ is a $(w/A, \tau)$-irreducible measure on $H$.
    
    By the $(\sigma, K_0, 1-\varepsilon)$-thin property of $G_{ij}$, we know that $\pi_{x_i}(\mu_j|_{G_{ij}|x_i})$ is $(C_1 K_0, \sigma)$-Frostman for some $C_1 := C_1(\sigma,\varepsilon)>0$. Indeed, for any ball $B(y, \delta) \subset \R^{n-1}$ either $B(y, \delta) \cap \pi_{x_i}(G_{ij}|_{x_i})$ is empty or there is a point $x_j \in \pi^{-1}_{x_i}(B(y, \delta)) \cap X_j $. In the latter case, there exists a constant $A=A(C,n)$ such that $\pi_{x_i}^{-1}(B(y, \delta)) \cap X_j \subset V_{x_i,y}(A \delta)$, where $V_{x_i,y}$ is the line through $x_i,y$. It follows that $\pi_{x_i}(\mu_j|_{G_{ij}|x_i})$ is $(C_1 K_0,\sigma)$-Frostman with $C_1 = A^{\sigma}$.

    It follows that the measures $\{\pi_{x_i}(\mu_j|_{G_{ij}|x_i})\}_{j:j\neq i}$, satisfy the conditions of Theorem \ref{theorem:section4main} with parameters
    $$
    k' = k-1, ~ n' = n-1, ~ s' = \sigma, ~ C' = C_1K_0, ~ w' = w/A, ~\tau' = \tau.
    $$
    Thus, given $j\neq i$, by our inductive hypothesis and $x_i \in \tilde{X}_i$ (see \eqref{eq:big-projection}), for all $\sigma' < \sigma$, $(\pi_{x_i}(\mu_j|_{G_{ij}|x_i}))_{j:j\neq i}$, have $(\sigma', K_1, 1-\varepsilon^{1/2})$-thin $(k-1)$-planes for some $K_1 := K_1(\varepsilon, s', \sigma', n-1, C_1K_0, w'))$. That is, there exists a Borel set
    $$
    \widetilde G_{\hat i}^{(x_i)} \subset \prod_{j:j\neq i} \pi_{x_i}(X_j)
    $$
    of measure at least $1-\varepsilon^{1/2}$ with respect to the measure $\prod_{j:j\neq i} \pi_{x_i}(\mu_j|_{G_{ij}|x_i})$ such that
    \begin{equation}\label{eq1}
    \pi_{x_i}(\mu_j|_{G_{ij}|x_i}) (V'(\delta)) \le K_1 \delta^{\sigma'} 
    \end{equation}
    holds for any affine subspace $V' = \psi_{k-1}(x_{\hat i})$ with $(x_{\hat i}) := (x_0, \dotsc, x_{i-1}, x_{i+1}, \dotsc, x_k) \in \widetilde G_{\hat i}^{(x_i)}$. For our purposes, equation \eqref{eq1} is not exactly what we need to conclude our results: we would like this condition to hold with respect to the measures $\{\pi_{x_i}(\mu_j)\}$ instead of $\{\pi_{x_i}(\mu_j|_{G_{ij}|x_i})\}$. This can be achieved by the following.
    \begin{prop}\label{prop:thin-planes-restriction}
        Let $\mu_0, \dotsc, \mu_k$ be arbitrary probability measures in $C$-good position in $\R^n$. Denote $X_i = \supp(\mu_i)$ and suppose that there are subsets $X'_i \subset X_i$ and $G' \subset X'_0 \times \dotsb \times X'_k$ such that:
        \begin{itemize}
            \item $\mu_i(X_i') \ge 1-\varepsilon$, $i=0, \dotsc, k$,
            \item $\mu_0 \times \dots \times \mu_k(G') \ge 1-\varepsilon$,
            \item the measures $\mu_0|_{X'_0}, \dots, \mu_k|_{X'_k}$ have $(\sigma, K, 1-\varepsilon)$-thin $(k-1)$-planes with respect to $G'$.
        \end{itemize}
        Then there exists some constant $C_1:= C_1(\varepsilon, \sigma, n, C)$ 
        such that 
        $\mu_0, \dotsc, \mu_k$ have $(\sigma-\varepsilon, C_1K, 1-2\varepsilon)$-thin $(k-1)$-planes with respect to some Borel $G \subset G'$.
    \end{prop}

    \begin{proof}[Proof of Proposition \ref{prop:thin-planes-restriction}]
    Let $C_1>0$ to be determined later, and let $V_{x_0, \dotsc, x_k}$ denote the $k$-plane spanned by $x_0, \dotsc, x_k$. For dyadic $\delta>0$, define
    $$
    E_i(\delta) = \{(x_0, \dotsc, x_k) \in G':~ \mu_i(V_{x_0, \dotsc, x_k}(\delta)) > C_1 K \delta^{\sigma-\varepsilon}\}.
    $$
    Fix a tuple $(x_1, \dotsc, x_k) \in X_1 \times \dotsb \times X_k$ and let us bound $\mu_0(E_0(\delta)|_{(x_1, \dotsc, x_k)})$. Let $I$ be a maximal $\delta$-separated subset of $E_0(\delta)|_{(x_1, \dotsc, x_k)}$ such that
    the collection $\{V_{x_0, x_1, \dotsc, x_k}(\delta) \cap X_0\}_{x_0\in I}$ is pairwise disjoint. Thus, there exists a constant $A:=A(n,C)$ such that by the maximality of $I$ and the fact that $\mu_0, \dotsc, \mu_k$ are in $C$-good position, for any $x_0' \in E_0(\delta)|_{(x_1, \dotsc, x_k)}$ there is $x_0 \in I$ such that $x_0' \in V_{x_0, x_1, \dotsc, x_k}(A\delta)$ (otherwise we can add $x_0'$ to $I$ and increase its size, contradicting the maximality). So the neighborhoods $\{V_{x_0, x_1, \dotsc, x_k}(A\delta)\}_{x_0 \in I}$ cover the whole set $E_0(\delta)|_{(x_1, \dotsc, x_k)}$ and we obtain
    $$
    \mu_0(E_0(\delta)|_{(x_1, \dotsc, x_k)}) \le \sum_{x_0 \in I} \mu_0(G'|_{x_1, \dotsc, x_k} \cap V_{x_0, x_1, \dotsc, x_k}(A\delta)) \le |I| K (A\delta)^\sigma 
    $$
    by the thin $(k-1)$-planes property of $G'$. 
    
    By assumption, the intersections $\{X_0 \cap V_{x_0, x_1,\dotsc, x_k}(\delta)\}_{x_0 \in I}$
    are pairwise disjoint and so
    $$
    1 = \mu_0(X_0) \ge \sum_{x_0\in I} \mu_0(V_{x_0, x_1,\dotsc, x_k}(\delta)) \ge C_1 K \delta^{\sigma-\varepsilon} |I|.
    $$
    since $I \subset E_0(\delta)|_{(x_1, \dotsc, x_k)}$. 
    We conclude that $|I| \le (C_1K)^{-1} \delta^{\varepsilon-\sigma}$. Thus, 
    $$
    \mu_0(E_0(\delta)|_{(x_1, \dotsc, x_k)}) \le C_1^{-1} A^\sigma \delta^{\varepsilon}
    $$ 
    holds for any $(x_1, \dots, x_k) \in X_1\times \dots \times X_k$ and any dyadic $\delta > 0$.
    So if we define 
    $$
    E_0 = \bigcup_{\delta>0} E_0(\delta)
    $$
    then it follows by Fubini's theorem that 
    $$
    \mu_0\times \dots \times \mu_k(E_0) \le \sum_{0<\delta <1\ \text{dyadic}} C_1^{-1} A^\sigma \delta^{\varepsilon} < \varepsilon/(k+1),
    $$
    if we choose $C_1 := C_1(\sigma,C,n,\varepsilon)$ large enough. Similarly we get 
    $$
    \mu_0\times \dots \times \mu_k(E_i) <\varepsilon/(k+1),
    $$
    for $E_i = \bigcup_{\delta >0} E_i(\delta)$. Let 
    $$
    G = G' \setminus (E_0 \cup \dots \cup E_k),
    $$
    we claim that this $G$ resolves the Proposition. Clearly $\mu_0\times \dots \times \mu_k(G) \ge 1-2\varepsilon$ so we only need to check the thin planes property. Indeed, let $(x_0, \dotsc, x_k) \in G$ and pick a dyadic $\delta > 0$ and some index $i \in \{0,\dotsc, k\}$. Since $(x_0, \dotsc, x_k) \notin E_i(\delta)$ by design, we get
    $$
    \mu_i(V_{x_0, \dotsc, x_k}(\delta)) \le C_1 K \delta^{\sigma-\varepsilon},
    $$
    as desired. So $\mu_0, \dotsc, \mu_k$ have $(\sigma-\varepsilon, K', 1-2\varepsilon)$-thin planes with respect to the graph $G$ where $K' = C_1K$.
    \end{proof}

    Using Proposition \ref{prop:thin-planes-restriction} in our situation, we can find Borel sets $G^{(x_i)}_{\hat i}\subset \widetilde G_{\hat i}^{(x_i)}$ such that the measures $(\pi_{x_i}(\mu_j))_{j\neq i}$ have $(\sigma'-\varepsilon, C_1K_1, 1-2\varepsilon^{1/2})$-thin $(k-1)$-planes with respect to $G^{(x_i)}_{\hat i}$ for all $\sigma' < \sigma$.
    
    With this in mind, we can proceed to constructing a graph witnessing thin $k$-planes for the original measures $\mu_0, \dotsc, \mu_k$. We do this in a few steps. For each dyadic $\delta>0$, let $\mathcal D_\delta$ denote the collection of canonical $\delta$-cubes. Let $\delta > 0$ be a dyadic number and let $Q \in \mathcal{D}_\delta$ be a dyadic $\delta$-cube such that $\widetilde X_i \cap Q \neq \emptyset$. Fix an intersection point $x_i = x_i(Q) \in \widetilde X_i \cap Q$ and define
    \begin{equation}
        G_i(Q) = Q \times \left ( \pi_{x_i(Q)}^{-1} (G_{\hat i}^{(x_i(Q))}) \cap \prod_{j\neq i} G_{ij}|_{x_i(Q)}  \right)
    \end{equation}
    with the understanding that the first term $Q$ goes into the $i$-th position of the product. Next, we define
    $$
    G_i(\delta) = \bigcup_{Q:\ Q \cap \widetilde X_i \neq \emptyset} G_i(Q) 
    $$
    and we put 
    $$
    G_i = \bigcap_{\delta > 0} G_i(\delta).
    $$
    Finally, let 
    \begin{equation}\label{eq:def-G}
    G = G_0 \cap \dots \cap G_k.
    \end{equation}
    We claim that the constructed set $G$ witnesses $(\sigma' -  \varepsilon, K, 1-C_2 \varepsilon^{1/2})$-thin $k$-planes for the measures $(\mu_0, \dotsc, \mu_k)$ for all $\sigma'< \sigma$ where $K$ and $C_2$ are to be determined later. Indeed, first note that $G$ is clearly a Borel set. Furthermore, there exists a dimensional constant $C := C(n)$ such that for each dyadic $\delta > 0$, the set $G_i(\delta)$ is contained in a $C\delta$-neighborhood of $X_0 \times \dots \times X_k$ and so since $X_i$ are compact (as they are the supports of the measures $\mu_i$), we conclude that $G_i$ is contained in $X_0 \times \dotsb \times X_k$ for each $i=0, \dotsc, k$. We conclude that $G \subset X_0\times \dotsb \times X_k$. 

    Next, we verify that the measure of $G$ with respect to $\nu = \mu_0\times \dotsb \times \mu_k$ is large enough. From our construction, for $\delta >0$ and a dyadic $\delta$-box $Q$ we have
    $$
    \nu(G_i(Q)) = \mu_i(Q) \cdot \mu_{\hat i} \left ( \pi_{x_i(Q)}^{-1}( G_{\hat i}^{(x_i(Q))}) \cap \prod_{j\neq i} G_{ij}|_{x_i(Q)}  \right).
    $$
    Using that $G_{\hat i}^{(x_i(Q))}$ witnesses $(\sigma' - \varepsilon, C_1K', 1-2\varepsilon^{1/2})$-thin $(k-1)$-planes, we have
    $$
    \mu_{\hat i}\left ( \pi_{x_i(Q)}^{-1} G_{\hat i}^{(x_i(Q))}\right) = 
    \pi_{x_i(Q)}\mu_{\hat i}\left ( G_{\hat i}^{(x_i(Q))}\right) \ge 1- 2 \varepsilon^{1/2},
    $$
    and we use the fact that $x_i(Q) \in \widetilde X_i$, to bound
    $$
    \mu_{\hat i} \left (  \prod_{j\neq i} G_{ij}|_{x_i(Q)}  \right) = 
    \prod_{j\neq i} \mu_j( G_{ij}|_{x_i(Q)} ) \ge (1-\varepsilon^{1/2})^k \geq 1 - k\varepsilon^{1/2}.
    $$
    Putting these bounds together implies 
    $$
    \nu(G_i(Q)) \ge \mu_i(Q) (1-(k+2) \varepsilon^{1/2}).
    $$
    Summing over all dyadic $\de$-boxes $Q$ such that $Q\cap \widetilde X_i \neq \emptyset$, we obtain by \eqref{eq:big-projection}:
    $$
    \nu(G_i(\delta)) \ge \mu_i(\widetilde X_i) (1-(k+2) \varepsilon^{1/2}) \ge (1- \varepsilon^{1/2})(1 - (k+2) \varepsilon^{1/2}) \geq 1-(k+3)\varepsilon^{1/2}.$$
    
    By continuity, we get $\nu(G_i) \ge 1-(k+3) \varepsilon^{1/2}$ and therefore by \eqref{eq:def-G}, $$\nu(G) \geq 1-C_2\varepsilon^{1/2},$$
    where $C_2 = (k+1)(k+3)$ as desired.

    Last but not least we check the thin planes condition for $G$. It is enough to check the condition for any dyadic scale $\delta > 0$. Let $(x_0, \dotsc, x_k) \in G$ and fix an index $j \in \{0, \dotsc, k\}$. Let $V = \psi_{k}(x_0, \dotsc, x_k)$, we want to estimate the measure $\mu_j(V(\delta))$. Let $i$ be any index distinct from $j$. We have 
    $$
    (x_0, \dotsc, x_k) \in G \subset G_i \subset G_i(\delta).
    $$
    So there exists a dyadic $\delta$-box $Q$ such that $Q \cap \widetilde X_i \neq \emptyset$ and $x_i \in Q$. Then we have by definition
    $$
    (x_0, \dotsc, x_k) \in G_i(Q) = Q \times \left ( \pi_{x_i(Q)}^{-1}( G_{\hat i}^{(x_i(Q))}) \cap \prod_{t\neq i} G_{it}|_{x_i(Q)}  \right),
    $$
    i.e., we have $x_{t} \in G_{it}|_{x_i(Q)}$ for any $t \neq i$ and $(\pi_{x_i(Q)}(x_{\hat i})) \in G_{\hat i}^{(x_i(Q))}$.

    Next, we observe that if we let $V' = \psi_{k-1}(\pi_{x_i(Q)}(x_{\hat i}))$ then for a constant $A(n,C)$,
    \begin{equation}\label{eq3}
    \pi_{x_i(Q)}(V(\delta) \cap X_j) \subset V'(A \delta) \cap \pi_{x_i(Q)}(X_j).    
    \end{equation}
    The measures $\{\pi_{x_i(Q)}(\mu_t)\}_{t\neq i}$ have $(\sigma'-\varepsilon, C_1K_1, 1-2\varepsilon^{1/2})$-thin $(k-1)$-planes with respect to $G_{\hat i}^{(x_i(Q))}$. So we can use this property on the tuple $(\pi_{x_i(Q)}(x_{\hat i})) \in G_{\hat i}^{(x_i(Q))}$ at scale $A\delta$ to conclude that
    $$
    \pi_{x_i(Q)}(\mu_j) (V'(A \delta)) \le C_1K_1 (A\delta)^{\sigma'-\varepsilon}.
    $$
    Letting $K = C_1K_1A^{\sigma' - \varepsilon}$, by \eqref{eq3} and the definition of the pushforward, we conclude that $\mu_j(V(\delta)) \le K \delta^{\sigma' - \varepsilon}$ holds. Since $\sigma' < \sigma$ is arbitrary and $\varepsilon > 0$ is arbitrary, this concludes the proof.
\end{proof}

\subsection{Proof of Theorem \ref{thm:qualitative-irreducible-beck}} \label{sec:qualititative-irreducible-beck}
We now deduce the qualitative version of Beck's theorem for irreducible measures (Theorem \ref{thm:qualitative-irreducible-beck}). Let $\mu$ be an irreducible $(C,s)$-Frostman probability measure supported on $X$.

Our first lemma divides the measure $\mu$ into a list of measures in $C$-good position.

\begin{lemma}[Partitioning measure]\label{lem:partitioning-measure}
For each $C>0$ and $s > 0$, there exists $A(C)\ge 1$ such that the following hold.
    Let $s>0$, and suppose that $\mu$ is an $(C,s)$-Frostman probability measure that is irreducible in $\R^n$ and let $\supp(\mu) = X$. Then, for any $1\leq k\leq n$, there exist balls $B_0, \ldots, B_k$ such that $B_0\times \ldots \times B_k$ is $1/A(C)$-separated from $\mc S_k$ and $\mu(B_i) >0$ for $i=0, \ldots,k$. In particular, the restricted measures $\mu_i = \frac{1}{\mu(B_i)} \mu|_{B_i}$ are $(C/\mu(B_i), s)$-Frostman and irreducible.
\end{lemma}

\begin{proof}
    Since the measure is irreducible, its support $X = \supp\mu$ is not contained in any hyperplane. So we can find affinely independent points $x_0,\ldots, x_{k}\in X$. Then take $B_i$ to be the $\varepsilon$-ball around $x_i$ and note that $B_0\times \ldots \times B_k$ is $1/A(C)$-separated from $\mc S_k$ for a small enough $\varepsilon$. We clearly have $\mu(B_i) >0$ so we are done.
\end{proof}

The next lemma establishes the existence of a \emph{modulus of irreducibility} for  irreducible measures.

\begin{lemma}\label{lem:compactness-for-irreducible-measures}
    Let $\mu$ be an irreducible measure on $V\cap B(0,1)$ then for any $\tau >0$ there exists $w := w(\tau) >0$ such that $\mu$ is $(w, \tau)$-irreducible in $V$.
\end{lemma}

\begin{proof}
    This is a statement about compactness: for $j\ge 0$ let $H_j \subset V$ be subspace maximizing $\mu(H_j(2^{-j}))$. Let $H$ be a limiting point of the sequence $\{H_j\}$ (note that it sits in a compact subset of the Grassmannian). Then, by the irreduciblity of $\mu$,
    \[
    0 = \mu(H) = \limsup_{j\to \infty} \mu(H_j(2^{-j}))
    \]
    so there exists $j_0$ so that $\mu(H_j(2^{-j})) \le \tau\mu(B(0,1))$ for all $j\ge j_0$. 
\end{proof}

The outcome of Lemmas \ref{lem:partitioning-measure} and \ref{lem:compactness-for-irreducible-measures} is that for any $\tau > 0$, we have $(w,\tau)$-irreducible measures $\mu_0,\dots,\mu_k$ in $A(C)$-good position. By Theorem \ref{theorem:section4main}, $(\mu_0,\dots,\mu_k)$ span $\sigma$-thin $k$-planes.
Applying Lemma \ref{lem:thin-planes-implies-pushforward-frostman}, we deduce 
\[
\dim \mathcal P^k(X) \geq (k+1) \min\{\sigma, n-k\}
\]
for all $\sigma < s$. Letting $\sigma \nearrow s$ concludes the proof of Theorem \ref{thm:qualitative-irreducible-beck}.

\section{Measures in stable position}\label{sec:stable-position}

We fix a collection of flats $V_1, \ldots, V_m \subset \R^n$. Denote $n_j = \dim V_j$. We will also have a collection of measures $\mu_{j, 1}, \ldots, \mu_{j, n_j}$ supported on $V_j \cap B^n(0,1)$. Recall the algorithm of Lemma \ref{lem:decomposition}, which decomposes a measure $\mu$ into irreducible pieces $\mu_{j,i}$ supported on proper subflats $V_1,\dots,V_m\subset \R^n$. The basic issue behind ensuring that we span hyperplanes when we choose a point from the support of each $\mu_{j,i}$ is that the dimension of the affine span can be less than $n-1$ or equal to $n$.

For a flat $V \subset \R^n$ let $\overline{V} = \operatorname{span}(V \times \{1\}) \subset \R^{n+1}$ denote the linearization of $V$ (note that $\dim \overline{V} = \dim V+1$). In what follows, it will be convenient to work with linearizations of flats.

 For $j\in [k]$ let $v_{j, 1}, \ldots, v_{j, n_j+1} \in \R^{n+1}$ be a fixed orthonormal basis of the linear space $\overline{V_j}$ and let $A_j = (v_{j,1}, \ldots, v_{j, n_{j}+1})$ be the corresponding $(n+1) \times (n_j+1)$ matrix. For a choice of an index set $I_j \subset [n_j]$ for every $j=1, \ldots, m$ and
 $\overline{x} = (x_{j, i}:~j\in [m], ~ i\in I_j)$, where $x_{j,i} \in \supp \mu_{j,i}$, we denote $\overline I = \{(j, i), ~ j\in[m],i\in I_j\}$ and let $B_{\overline{I}}(\overline{x})$ be the $(n+1)\times (\sum |I_j|)$ matrix composed of vectors $\begin{bmatrix}x_{j, i}\\ 1\end{bmatrix} \in \overline{V_j}$ over all $(j, i) \in \overline{I}$. Since $n_j \le n$ for all $j$, we always have $\overline{I} \subset [m]\times [0,n]$. 

 For $J \subset [k]$ we let $A_J = (A_{j}, ~j\in J)$ which is a $(n+1)\times (\sum_{j\in J} (n_j+1))$ matrix. 
 For index sets $\overline{I}$ and $J \subset [k]$, we will be interested in the properties of the matrices
 \[
 (B_{\overline{I}}(\overline{x}), A_{J})
 \]
 over various choices of $\overline{x}$ supported on our measures $\mu_{j, i}$. For a matrix $A$ and an integer $r\ge 1$ we let $M_r(A)$ be the maximum absolute value of an $r \times r$ minor of $A$.

\begin{definition}
    Given a collection of flats $V_j$ and bases $v_{j, i}$ as above, we say that a collection measures $(\mu_{j, i}:~ i\in [n_j], j\in [m])$ is in $c$-stable position if for every choice of index sets $\overline{I}$ and $J$ there exists a number $r=r(\overline{I}, J)$ such that we have $\operatorname{rank} (B_{\overline{I}}(\overline{x}), A_J)=r$ and $M_r(B_{\overline{I}}(\overline{x}), A_J) \ge c$ for any $\overline{x} \in \prod_{j, i} \supp \mu_{j,i}$.
\end{definition}

\begin{lemma}\label{lem:stable-position}
    Let $V_1, \ldots, V_m$ be arbitrary flats, set $n_j=\dim V_j$. Let $\mu_{j, i}$, $i=1, \ldots, n_j$, be arbitrary irreducible measures on $V_j \cap B^n(0,1)$. Then there exist $c>0$ and balls $B_{j, i} \subset V_j \cap B^n(0,1)$ such that $B_{j, 1}\times \ldots\times B_{j, n_j}$ is $c$-separated from $\mc S_{n_j}$,  $\mu_{j, i}(B_{j, i}) >0$ and the collection $(\mu_{j, i}|_{B_{j, i}})$ is in $c$-stable position.
\end{lemma}

\begin{proof}
    Observe that $M_r(\cdot)$ is a continuous function, so if we have $M_r(B_{\overline{I}}(\overline{x}), A_J) \ge c>0$ for some $\overline{x}$, then we have $M_r(B_{\overline{I}}(\overline{x'}), A_J) \ge c/2$ for all $\overline{x}'$ within a small neighborhood of $\overline{x}$. Now for each $\overline{x} \in \prod \supp\mu_{j,i}$ consider the tuple \[
    \overline{r}(\overline{x}) = (\operatorname{rank}(B_{\overline{I}}(\overline{x}), A_J):~ \text{ all possible }\overline{I}, J).
    \]
    Note that $\overline{r}(\overline{x})$ is a vector of integers between $0$ and $n+1$ in some finite dimensional vector space (there are only finitely many options for $\overline{I}$, $J$). So we can find a vector $\overline{x} \in \prod \supp\mu_{j,i}$ for which $\overline{r}(\overline{x})$ has maximum possible sum of coordinates. So for any other $\overline{x'}$ we have $\|\overline{r}(\overline{x'})\|_1 \le \|\overline{r}(\overline{x})\|_1$. On the other hand, by the continuity, we must have $\overline{r}(\overline{x'}) \ge \overline{r}(\overline{x})$ (coordinate-wise) for $\overline{x'}$ in a small neighbourhood of $\overline{x}$. So we can take $B_{j, i}$ to be small balls around $x_{j,i}$ and $c=\frac{1}{2}\min_{\overline{I}, J} M_{\operatorname{rank}(B_{\overline{I}, J}(\overline{x}), A_J)}(B_{\overline{I}, J}(\overline{x}), A_J)$. Finally note that $c$-stability implies the separation from $\mc S_{n_j}$.
\end{proof}

We will need a lemma showing that $c$-stable position is preserved under radial projections.

\begin{lemma}\label{lem:projecting-stable}
    Suppose that $V_1, \ldots, V_m$ are subspaces in $\R^n$ and $(\mu_1, \ldots, \mu_m)$ are measures in $c$-stable position with respect to them. Let $I_0 \subset [m]$ be an index set and let $n_0 = \dim \langle\overline{x}_{I_0} \rangle =\dim \langle x_{i}, ~ i\in I_0\rangle$ for arbitrary $x_i \in \supp\mu_i$. Let $U \subset \R^n$ be a subspace of dimension $n-n_0-1$ transversal to $\langle\overline{x}_{I_0}\rangle$ (for all $\overline{x}$) and consider the projection map $\pi_{\langle\overline{x}_{I_0}\rangle}: \R^n \setminus \langle\overline{x}_{I_0}\rangle \to U$.
    Define
    \[
    \theta = \angle( \overline{U}, \overline{\langle\overline{x}_{I_0}\rangle}).
    \]
    Then the collection of measures $\tilde \mu_i = \pi_{\langle \overline{x}_{I_0}\rangle}(\mu_i)$, $i\in [m]\setminus I_0$ is in $\gtrsim (\theta c)^{O(1)}$-stable position with respect to spaces $\pi_{\langle \overline{x}_{I_0}\rangle}(V_i)$, $i \in [m]\setminus I_0$. 
\end{lemma}

\begin{proof}
    Choose an orthonormal basis in the linearization $\overline{\pi_{\langle \overline{x}_{I_0}\rangle}(V_i)} \subset \overline{U}$ for $i \in [m]\setminus I_0$ and define the corresponding matrix $A'_i$. Choose some $I,J \subset [m]\setminus I_0$ and pick any $\overline{x'} \in \prod_{i\in [m]\setminus I_0} \supp\mu_i$. Let $r+r_0$ denote the rank of the matrix $(B_{I_0\sqcup I}(\overline{x}_{I_0},\overline{x'}), A_J)$ (where $r_0 = \operatorname{rank}(B_{I_0}(\overline{x}))$ notably only depends on $I_0$ by $c$-stability),  and furthermore by the $c$-stability of the collection $(\mu_i)$ we have
    \[
    c \le M_{r+r_0}(B_{I_0\sqcup I}(\overline{x}_{I_0}, \overline{x'}), A_J) = M_{r+r_0}(B_{I_0}(\overline{x}), B_I(\overline{x'}), A_J).
    \]
    Note that $\overline{\langle \overline{x}_{I_0}\rangle}$ and $\overline U$ are complementary subspaces of $\R^{n+1}$. Thus, by subtracting appropriate linear combinations of columns in $B_{I_0}(\overline{x})$ we can make the columns of $B_I(\overline{x'})$ and $A_J$ belong to $\overline{U}$. The coefficients of these linear combinations are upper bounded by $(c\theta)^{-O(1)}$ (since $ \overline{x}_{I_0}$ is in $c$-stable position inside $\langle \overline{x}_{I_0}\rangle$ and this space has angle $\theta$ with $\overline{U}$). The result of the column operations will give us precisely the linearization of the radial projection map $\pi_{\langle \overline{x}_{I_0}\rangle}$, and so we obtain
    \[
    (c\theta)^{O(1)} \le M_{r+r_0}(B_{I_0}(\overline{x}), B'_I(\pi_{\langle \overline{x}_{I_0}\rangle}(\overline{x'})), A'_J) \lesssim M_{r_0}(B_{I_0}(\overline{x})) M_r( B'_I(\pi_{\langle \overline{x}_{I_0}\rangle}(\overline{x'}) ), A'_J),
    \]
    and so we conclude that $M_r( B'_I(\pi_{\langle \overline{x}_{I_0}\rangle}(\overline{x'}) ), A'_J) \gtrsim (c\theta)^{O(1)}$. Similarly, we have 
    \[
    r+r_0 = \operatorname{rank}(B_{I_0\sqcup I}(\overline{x}_{I_0}, \overline{x'}), A_J) = \operatorname{rank}(B_{I_0}(\overline{x}))+\operatorname{rank}( B'_I(\pi_{\langle \overline{x}_{I_0}\rangle}(\overline{x'}) ), A'_J)
    \]
    and so $r = \operatorname{rank}( B'_I(\pi_{\langle \overline{x}_{I_0}\rangle}(\overline{x'}) ), A'_J)$. This completes the proof. 
\end{proof}

Using the lemmas of this section, we can reduce our situation to the case when the measures are in $c$-stable position for some $c>0$. This significantly simplifies the geometric analysis and reduces the problem to understanding the quantities $r(\overline{I}, J)$.

\section{Reducing Theorem \ref{thm:main} to the minimal case} \label{sec:reduction-to-minimal}

Let us first recall our current set up of the problem from the previous sections. From Section \ref{sec:irreducibledecomp} we have found flats $V_1,\dots, V_m$ and corresponding irreducible $s$-Frostman measures $\mu_1,\dots, \mu_m$. Furthermore, by the irreducible case in Section \ref{sec:irreducible implies saturated}, we know that for each $(V_j,\mu_j)$, there exists measures $\mu_{j,i}$ for $1 \leq j \leq \dim V_j$ spanning thin $(\sigma, K, 1-\varepsilon)$-thin codimension 1-planes for all $0 < \sigma < s < n-1$ and $\varepsilon>0.$ In fact, we know more about the collection of our flats $\{V_1, \dots,V_m\}$ from Lemma \ref{lem:decomposition}.

\begin{definition}[\textbf{NC} flats] \label{def:NCFlats}
    Let $V_1,\dots, V_m \subset\R^n$ be a collection of affine flats. We say that the collection $V_1,\dots, V_m$ is \emph{non-concentrated}, or \textbf{NC}, if any list of flats $F_1,\dots, F_r$ with $\bigcup_{j \in [m]} V_j \subset \bigcup_{i\in[r]} F_i$ satisfies 
    \[
    \sum_{i\in [r]} \dim F_i \geq n.
    \]
    Note that this implies $\sum_{j \in [m]} \dim V_j \geq n.$
\end{definition}

By combining several of our \textbf{NC} flats $V_j$ into a single flat $F_J =\aff(V_j : j \in J)$ for $J \subset [m]$, and showing that on each $F_J$ we have thin codimension 1-planes, we can in fact assume that our collection $\{V_1,\dots, V_m\}$ is minimal in a certain sense.

\begin{definition}[Minimal Collections of Flats] \label{def:minimal}
    Given a collection of flats $F_1,\dots, F_k\subset \R^n$ and $J \subset [k]$, let $F_J := \aff(F_j : j \in J)$. Then, we say that $F_1,\dots, F_k$ is \emph{minimal} if $\dim F_{[k]} = n \leq \sum_{j \in [k]} \dim F_j$ and $\dim F_J \geq \sum_{j\in J} \dim F_j$.
\end{definition}

To this end, let us now state the main results that will be the contents of the rest of the paper, and show we may assume we have a minimal collection of flats.

 We will prove the following theorem.

\begin{theorem}[General case]\label{thm:general-case}
    Let $V_1, \ldots, V_m \subset \R^n$ be an {\bf NC} collection of flats and suppose that $(\mu_{j, i}:~ j\in [m], i\in [n_j])$ are $(C, s)$-Frostman measures in $c$-stable position. Suppose that $\mu_{j, i}$ is $(w, \tau)$-irreducible inside $V_j$ for every $j, i$. 

    There exist some $n'_j \le n_j$ with $\sum n'_j = n$ such that the following holds. 
    For any $\varepsilon>0$ and $\sigma< s$, if $\tau < \tau(\varepsilon,s, \sigma, C, c, n)$ and $K= K(\varepsilon,s, \sigma, C, c, n, w)$, then the measures $(\mu_{j, i}, ~ j\in [m], i \in [n'_j])$ have $(\sigma, K, 1-\varepsilon)$-thin hyperplanes. 
\end{theorem}

\begin{remark}
    We start with a qualitatively $\mathbf{NC}$ measure $\mu$, from which we get measures $\mu_{j,i}$, each with modulus of irreducibility $\tau_{j,i}(w)\to 0$ as $w\to 0$. We take $w$ sufficiently small so that $\tau_{j,i}(w) < \tau(\varepsilon, s, \sigma, C, c, n)$ to apply the general case theorem.
\end{remark}

At the end of this section we deduce Theorem \ref{thm:main}, that if $X$ is an $\mathbf{NC}$ set, then
\[
\dim \mathcal P^{n-1}(X) \ge n\min\{\dim X, 1\}
\]
from the more quantitative Theorem \ref{thm:general-case}.

We prove Theorem \ref{thm:general-case} using an inductive scheme.

We denote $V_J = \aff(V_j, ~j\in J)$ for $J\subset [m]$ and let $n_J = \dim V_J$. 

\begin{definition}[Saturated set of indices.]
    Let $V_1, \ldots, V_m \subset \R^n$ be a collection of flats. A set of indices $I \subset [m]$ is called {\em saturated with respect to $V_1, \ldots, V_m$} if there exist functions $K_I(\varepsilon,s, \sigma, C, c, n, w)$ and $\tau_I(\varepsilon,s, \sigma, C, c, n)$ such that the following holds for any $\varepsilon >0$, $\sigma <s$, $w>0$, $\tau \le \tau_I$, $K \ge K_I$. 

    Let $\mu_{t, i}$ for $t\in I$ and $i \in [\dim V_t]$ be $(C, s)$-Frostman measures in $c$-stable position with respect to $V_1, \ldots, V_m$. Suppose that $\mu_{t, i}$ is $(w, \tau)$-irreducible in $V_t$. Then there exist some $n'_t \le \dim V_t$ with $\sum_{t\in I} n'_t = \dim V_{I}$ and such that the measures $(\mu_{t, i}:~ t\in I, ~i \in [n'_t])$ have $(\sigma, K, 1-\varepsilon)$-thin hyperplanes inside $V_I$. 
\end{definition}

Thus, our goal is to show that $[m]$ is saturated with respect to $V_1, \ldots, V_m$, whenever that is an \textbf{NC} collection of flats. The irreducible case, Theorem \ref{theorem:section4main}, implies that $\{j\}$ is saturated for every $j \in [m]$. Here is our main inductive claim:

\begin{theorem}[Minimal case]\label{thm:minimal-case}
    Let $V_1, \ldots, V_m \subset \R^n$ be a collection of flats. 
    Let $I_1, \ldots, I_k \subset [m]$ be a collection of pairwise disjoint, non-empty subsets and suppose that each $I_j$ is saturated.
    Let $F_j = V_{I_j}$ and suppose that $F_1, \ldots, F_k$ is a minimal collection of flats inside $\langle F_1, \ldots, F_k \rangle$. Then $I_1\sqcup \ldots \sqcup I_k$ is also saturated with respect to $V_1, \ldots, V_m$.
\end{theorem}

Now we show how to deduce the General Case from the Minimal Case.

\begin{proof}[Proof that Theorem \ref{thm:general-case} follows from Theorem \ref{thm:minimal-case}]
We say that a collection of disjoint subsets $J_1,\dots, J_k \subset [m]$ is \emph{minimal} if the collection of flats $F_j = V_{J_j}$ is minimal inside $V_{J_1\cup \dots \cup J_k}$ (see Definition \ref{def:minimal}).

\begin{lemma}\label{lem:minimal-case}
    Let $J_1, \ldots, J_k \subset [m]$ be pairwise disjoint non-empty sets. Suppose that $J_j$ is saturated for $j=1, \ldots, k$. Suppose that the collection of flats $F_j = V_{J_j}$ is minimal inside the flat $V_{J_1 \sqcup \ldots \sqcup J_k}$. Then $J_1\sqcup \ldots \sqcup J_k$ is saturated as well. 
\end{lemma}

\begin{proof}
    We identify $V_{J_1 \sqcup \ldots \sqcup J_k}$ with $\R^{n'}$ for $n' = \dim V_{J_1 \sqcup \ldots \sqcup J_k}$ and apply Theorem \ref{thm:minimal-case}. Note that the constants $C, c$ will only change by a constant factor after the change of coordinates and that the $c$-stable position with respect to $V_1, \ldots, V_m$ implies $\gtrsim_c1$-stable position with respect to spaces $V_{J_1}, \ldots, V_{J_k}$ in the new coordinates.
    Thus, the condition that $J_j$ is saturated can be used as an input in Theorem \ref{thm:minimal-case} and the thin-planes conclusion implies that $J_1\cup\ldots \cup J_k$ is also saturated. 
\end{proof}

    \begin{lemma}\label{PartitionIntoMinimalParts}
        Suppose $J_1\cup \cdots \cup J_t= [m]$ is a partition of $[m]$ with $t\geq 2$ such that $J_i$ is saturated for every $1\leq i \leq t$. Then, for some $I \subset [t]$ with $|I| \geq 2$, the set $(J_i : i \in I)$ is minimal. In particular,
        the set $\bigcup_{i \in I} J_i$ is saturated.
    \end{lemma}

    \begin{proof}[Proof of Lemma \ref{PartitionIntoMinimalParts}]
        If the collection $J_1,\dots, J_t$ is minimal then  $[m]$ is saturated by Lemma \ref{lem:minimal-case}. Hence, suppose $J_1,\dots, J_t$ is not minimal. This implies that either $\dim V_{[t]} > \sum_{j = 1}^t \dim V_{J_j}$ or $\dim V_I < \sum_{j \in I} \dim V_{J_j}$ for some proper $I \subsetneq [t].$ The former case is impossible as the collection $\{V_i\}$ is \textbf{NC} (take $F_j = V_{J_j}$ for all $1\leq j \leq t$ in Definition \ref{def:NCFlats}). Hence, the latter inequality must hold for some $I \subsetneq [t].$

        Thus, there must exist some minimal by inclusion set $I \subset [t]$ such that for $J_I = \bigcup_{j \in J} J_j$ we have $\dim V_{J_I} < \sum_{j \in I} \dim V_{J_j}$ while $\dim V_{J_{I'}} \geq \sum_{j\in I'} \dim V_{J_j}$ for any proper $I' \subsetneq I.$ This implies that $|I|\geq 2$ as the singleton sets are saturated by the irreducible case. Therefore, we have that the collection of sets $\{J_j : j \in I\}$ is minimal. So, by Lemma \ref{lem:minimal-case}, $J_I$ is saturated, completing the proof.
    \end{proof}

    We can now finish the proof by iterating Lemma \ref{PartitionIntoMinimalParts}. Start with the trivial partition $\{1\}\cup \dots \cup \{m\} = [m]$. Applying Lemma \ref{PartitionIntoMinimalParts}, we can merge some subset of elements in the above partition to get a new partition into saturated sets with at least $m-1$ parts. Reiterating Lemma \ref{PartitionIntoMinimalParts} implies that eventually we will arrive at the partition $[m] = [m]$ which implies that $[m]$ is saturated. Unraveling the definition of a subset being saturated gives the conclusion of the General Case.
\end{proof}

\begin{remark}
Note that while we in fact have irreducible measures $\mu_j$ on each $V_j$, in the Minimal Case we may not assume $\mu_j$ are irreducible in $F_j$ due to the merging process described above.
\end{remark}

The proof of the Minimal Case is the content of Section \ref{sec:minimal-case}. We now show that the General Case implies our main result (Theorem \ref{thm:main}).

\begin{proof}[Proof that the General Case implies Theorem \ref{thm:main}]
Let $X\subset \R^n$ be a Borel \textbf{NC} set. Then, the Decomposition Lemma (Lemma \ref{lem:decomposition}) gives us an \textbf{NC} collection of flats $\{V_j\}_{j \in [m]}$ with irreducible $s$-Frostman measures $\mu_j$ supported on $X\cap V_j$ for all $s< \dim X.$ Let $n_j = \dim V_j$. 

By Lemma \ref{lem:partitioning-measure}, we can find restrictions $\mu_{j,i} = \frac{1}{\mu_j(B_{j, i})} \mu_j|_{B_{j, i}}$ which are $s$-Frostman irreducible probability measures on $V_j$ in good position. By Lemma \ref{lem:stable-position}, we can further restrict the measures $\mu_{j, i}$ onto some smaller balls so that the collection $(\mu_{j, i}: ~i \in [n_j], j\in [m])$ is $c$-stable for some $c>0$ and $\mu_{j,i}$ is still an irreducible $s$-Frostman probability measure. For each $\mu_{j, i}$ we can define the modulus of irreducibility $\tau_{j, i}(w)$ as the maximum measure of a $w$-neighborhood of a proper subspace $H\subset V_j$. Then by irreducibility we have $\tau_{j, i}(w)\to 0$ as $w\to 0$ and $\mu_{j, i}$ is $(w, \tau_{j, i}(w))$-irreducible in $V_j$. So if we take $w$ sufficiently small, then all $\tau_{j, i}(w)$ become small enough to apply Theorem \ref{thm:general-case} and we conclude that for some $n'_j\le n_j$, $\sum_{j} n'_j = n$, the measures $(\mu_{j, i}, ~i \in [n'_j], j \in [m])$ have $(\sigma, K, 1-\varepsilon)$-thin hyperplanes. 
Call this subcollection of measures $\{\tilde{\mu}_i\}_{i = 1}^n$. Then, by Lemma \ref{lem:thin-planes-implies-pushforward-frostman}, it follows that the pushforward measure $\psi_{n-1}(\tilde{\mu}_1\times \cdots \times \tilde\mu_n)$ is a nonzero $n \sigma$-Frostman measure on $\mathcal A(n,n-1)$ and $\sigma \leq 1$. Notice that since $\supp \tilde\mu_i \subset X$ for all $i$, this implies that the pushfoward measure is supported on $\mathcal P^{n-1}(X)$. Given this holds for all $0 < \sigma < \min\{s,1\} < \dim X$ and arbitrary $\varepsilon>0$, this concludes the proof.
\end{proof}

\section{The minimal case} \label{sec:minimal-case}

In this section we prove Theorem \ref{thm:minimal-case}. We will prove the theorem in dimension $n'$ and use variable $n$ to denote something else. 
Let $V_1, \ldots, V_m \subset \R^{n'}$ be a collection of flats and fix $I_1, \ldots, I_k\subset [m]$ and $F_j = V_{I_j}$.
Suppose that $F_1, \ldots, F_k \subset \langle F_1, \ldots, F_k\rangle$ be a minimal collection of flats and denote $n = \dim \langle F_1, \ldots, F_k\rangle$. 
Let $n_j = \dim F_j$. Suppose that $\mu'_{t, i}$, $t\in [m]$, $i=1, \ldots, \dim V_t$ are measures on $V_t$ such that $(\mu'_{t, i}:~ i\in [\dim V_t], t\in [m])$ are in $c$-stable position with respect to $V_1, \ldots, V_m$ and that each $\mu'_{t, i}$ is $(w, \tau)$-irreducible inside $V_{t}$. By the assumption of the theorem we know that for every $j$ there are some numbers $d_t \le \dim V_t$, $t\in I_j$, so that $\sum_{t\in I_j} d_t = \dim F_j$ and the collection $(\mu'_{t, i}: t\in I_j, ~i\in [d_t])$ has $(\sigma, K_j, 1-\varepsilon)$-thin $(n_j-1)$-planes, provided that $\tau \le \tau_j(\varepsilon, s, \sigma, C, c)$ and $K_j = K_j(\varepsilon, s, \sigma, C, c, w)$. Denote by $G_j$ the corresponding thin-planes graph. 

Let us rotate the coordinate system so that $\langle F_1, \ldots, F_k\rangle = \R^n \times \{0\} \subset \R^{n'}$. We now drop the 0 and view these as flats in $\R^n$ (flats $V_{t}$ with $t \not \in \bigcup I_j$ will not be used). Note that after this operation the measures will still be in $\sim c$-stable position. For notational convenience let us label by $\mu_{j,  1}, \ldots, \mu_{j, n_j}$ the collection of measures $\mu_{t, i}$, $t\in I_j$, $i \in [d_t]$ (say in the increasing order of $t$-index and then increasing order of $i$-index). We will mostly work with this new notation and will not need the underlying spaces $V_t$ for a while,  but they will become relevant at the end of the argument. 

Following the notation from Section \ref{sec:stable-position}, we define matrices $(B_{\overline{I}}(\overline{x}), A_J)$, where now each $A_j= (f_{j,1}, \ldots, f_{j, n_{j}+1})$ corresponds to a fixed orthonormal basis $f_{j, i}$ of $\overline{F_j}$. 
Let $r(\overline{I}, J)$ denote the rank of the matrix $(B_{\overline{I}}(\overline{x}), A_J)$. Note that it does not depend on the choice of $\overline{x} \in \prod \supp\mu_{j, i}$ and we moreover have $M_{r(\overline{I}, J)}(B_{\overline{I}}(\overline{x}), A_J) \ge c$ for some $c>0$. We denote by $P_{\overline{I}}(\overline{x})$ the affine space spanned by vectors $x_{j,i}$ over $(j,i) \in \overline{I}$. We denote $F_J = \aff(F_j:~j\in J)$. 

First of all, we reduce to the case when $\sum_{j=1, j\neq j'}^k n_j \le n-1$ holds for every $j'\in [k]$. Indeed, otherwise note that the subcollection $\{F_j, ~ j\neq j'\}$ is still minimal and so we can restrict to this collection. Write $\sum_{j=1}^k n_j = n+p$ for some $p\ge 0$. We then conclude that $n_j \ge p+1$ for every $j=1, \ldots, k$. 

The next proposition records the fact that $r(\overline{I}, J)$ is the dimension of the linear space $\langle \overline{P_{\overline{I}}(\overline{x})}, \overline{F_J}\rangle$.

\begin{prop}\label{prop:minus1}
    For index sets $\overline{I}, J$ and any $j\in [k]$ and $ i\in [n_j]$:
    \[
    r(\overline{I} \cup \{(j, i)\}, J) \le r(\overline{I}, J) + 1,
    \]
    \[
    r(\overline{I}, J \cup \{j\}) \le r(\overline{I}, J) + n_j+1.
    \]
    \[
    r(\overline{I}\cup (\{j\}\times [n_i]), J) \ge r(\overline{I}, J \cup \{j\}) - 1.
    \]
\end{prop}

\begin{proof}
    The first two statements are clear from the definitions of ranks. For the last one notice that $\overline{P_{\{j\} \times [n_j]}(\overline{x})}$ is a codimension 1 subspace in $F_j$.
\end{proof}

\subsection{\texorpdfstring{Special case $p=0$}{Special case p=0}} \label{sec:proofofkeylemma}
We will first consider the special case when $p=0$, i.e. $\sum_j n_j = n$. The argument simplifies significantly in this case. Later we consider the general case $\sum_j n_j = n+p$ which will require an additional careful inductive argument.

For a subset $I \subset [k]$ and $J \subset [k]$ let us denote 
\[
r(I, J) = r(\overline{I}, J), ~ \text{ where } \overline{I}=\bigcup_{j\in I} \{j\}\times [n_j].
\]
With this notation the last line in Proposition \ref{prop:minus1} gives 
\begin{equation}\label{eq:minus1}
    r(I\cup \{j\}, J) \ge r(I, J \cup \{j\}) - 1.
\end{equation}

For $I \subset [k]$, we also denote
\[
\overline {P_I(\overline x)} = \overline{P_{\overline I}(\overline x)},\qquad \overline I = \bigcup_{i\in I}\{i\}\times[n_i].
\]
and $J\subset[k]$,
\[
\overline{ P_{{I},J}(\overline x)} = \mathrm{span}(\overline{P_{I}(\overline x)}, \overline{F_j},j\in J)
\]
with similar notation for $\overline{P_{I,J}(\overline x)}$.
With this notation, $r(I, J)$ is the dimension of the subspace $\overline{P_{I, J}(\overline x)}$ for every $\overline x\in G$.

\begin{prop}[cf. \cite{do2018extending} Claim 2]\label{prop:induction-rank}
     Let $F_1,\dots, F_k \subset \R^n$ be a minimal collection of flats with $\sum n_j = n$. For any disjoint $I, J \subset [k]$
    \[
    r(I, J) = \begin{cases}
    n_I, ~ \text{ if }J=\emptyset,\\
    \ge n_{I \cup J}+1, ~\text{ otherwise}.
    \end{cases}
    \]
\end{prop}

\begin{proof}
    We prove the statement for $I \subset [s]$ using induction on $s$. If $s=0$ then $I=\emptyset$, in this case $r(\emptyset, J)$ coincides with the dimension of the space $\overline{F_J }$. 
    So $r(\emptyset, J) \ge n_{J}+1$ follows from the assumption that $F_1, \ldots, F_k$ is a minimal collection of flats. 

    Now let $s\ge 1$ and $I \subset [s]$ be a set containing the element $s$. If $J = \emptyset$ then 
    \[
    r(I, \emptyset) \ge r(I\setminus \{s\}, \{s\}) - 1 \ge n_{I},
    \]
    by inductive hypothesis and (\ref{eq:minus1}). So we may assume $J \neq \emptyset$. 

    Choose any $x_{j, i} \in  \supp\mu_{j,i}$ for $j \in I$ and $i=1, \ldots, n_j$ and let $\overline{P_I} = \overline{P_{I}(\overline{x})}$. By definition, we have $r(I, J) = \dim \langle \overline{P_I}, \overline{F_J} \rangle$. Note that if 
    \[
    \langle \overline{P_{I\setminus s}}, \overline{F_J}\rangle \cap \overline{F_s} = 0,
    \]
    then we get $r(I, J) = r(I\setminus\{s\}, J)+ r(\{s\}, \emptyset) \ge (n_{I\cup J}+1-n_s) + n_s \ge n_{I\cup J}+1$ by induction. 
    Now let $W = \langle P_{I\setminus s}, F_J\rangle \cap F_s$, by the above we may assume that $W\neq\emptyset$ (perhaps after performing a mild projective transformation). 

    Let $w\in W$ any point and let $\nu = \delta_w$ be the Dirac $\delta$-mass at $w$. By Lemma \ref{lem:swapping}, the $(\prod \mu_{s,i})$-measure of the set of $\overline{x_s} \in G_s$ for which we have $w \in P_s(\overline{x}_s)(\delta)$ is upper bounded by $\varepsilon$ where $\varepsilon\to 0$ as a function of $\delta\to0$. In particular, for almost every $\overline{x_s}$ we have $W\not \subset P_{s}(\overline{x}_s)$. Thus we also get $\overline{W}\not \subset\overline{ P_{s}(\overline{x}_s)}$ and so since $P_{s}(\overline{x}_s)$ is a hyperplane in $F_s$:
    \[
    \langle \overline{P_{I\setminus s}}, \overline{F_J}, \overline{P_{s}(\overline{x}_s)}\rangle = \langle \overline{P_{I\setminus s}}, \overline{F_J}, \overline{F_s}\rangle 
    \]
    Since the rank $r(I, J)$ does not depend on which $\overline{x}$ we fix, it follows that 
    \[
    r(I, J) = r(I\setminus \{s\}, J\cup \{s\}) \ge n_{I\cup J}+1,
    \]
    by induction. 
\end{proof}

Proposition \ref{prop:induction-rank}, implies that
\begin{equation}\label{eqn:RobustRankKeyLemma}
r([k], \emptyset) = n \quad \text{and} \quad r([k]\setminus \{j\}, \{j\}) = n+1
\end{equation}
for all $j \in [k]$, where recall $r(I,J) = r(\overline I,J)$.  

Now we use this to construct a thin-planes graph for the collection $(\mu_{j, i}:~j\in [k],~i \in [n_j])$ (recall that we are in the case $\sum n_j=n$). Define $G = \prod_{j} G_j$ where $G_j$ is a $(\sigma, K_j, 1-\varepsilon)$-thin planes graph for $(\mu_{j, i}:~i\in [n_j])$. 
We will use \eqref{eqn:RobustRankKeyLemma} together with $\sum n_j = n$ to show that for $\overline x \in G$, $V_{\overline x}(\delta)$ intersects each $F_j$ in the $C'\delta$-neighborhood of $V_{\overline x_j}$ for some constant $C'(n,c)$. This will allow us to apply the thin-planes property of the measures $(\mu_{j, 1},\dots,\mu_{j, n_j})$ on $F_j$.

    \begin{prop}\label{prop:nice1}
The following hold for every $\overline x \in G$ and every $j\in[k]$.
\begin{enumerate}
\item $\overline{F_j}\cap \overline{P_{[k] \setminus \{j\}}(\overline{x}}) = \{0\}$.

    \item $\overline{F_j}\cap \overline{P_{[k]}(\bar x)} = \overline{P_{j}(\bar x)}$.
    
    \item The smallest principal angle between flats satisfies $\angle(\overline{F_j}, \overline{P_{[k] \setminus \{j\}}(\overline{x}})) \gtrsim_n c$.
\end{enumerate}
\end{prop}
\begin{proof}
For $\overline x$ fixed, and $I\subset [k]$, we denote
\[
B_I = B_I(\overline x) = B_{\bigcup_{j\in I}\{j\}\times [n_j]}(\overline x),
\]
and for $J\subset[k]$, we denote $\overline{P_{ I}} = \overline{P_{I}(\overline x)}.$

To prove the claims, we will use the fact that, by definition, $\overline{P_{I}}$
is the column space of $B_{I}$,  and $\overline F_j$ is the column space of the matrix $A_{j}$.

First we prove (1). Since $r([k]\setminus\{j\},\{j\}) = n+1$, we have $\mathbb R^{n+1} = \overline{P_{[k]\setminus\{j\}}}  \oplus \overline F_j$, so in particular (1) holds. 

To prove (2), by Proposition \ref{prop:induction-rank}, we have $r([k],\emptyset) = n$, $r(\{j\},\emptyset) = n_j$, and $r([k]\setminus\{j\},\emptyset)=n-n_j$.  Thus the corresponding dimension statements hold: $\dim \overline{P_{[k]}} = n$, $\dim \overline{P_{j}}= n_j$, $\dim \overline{P_{[k] \setminus \{j\}}} = n-n_j$.  By the dimension-sum formula for subspaces of $\R^{n+1}$, $\overline{P_{[k]}} = \overline{P_{\{j\}}} \oplus \overline{P_{[k] \setminus \{j\}}}$ is therefore a direct sum. Hence,
\[
\overline{F_j}\cap \overline{P_{[k]}} = \overline{F_j} \cap \overline{P_{j}}+\overline{F_j}\cap \overline{P_{[k] \setminus \{j\}}} = \overline{F_j}\cap \overline {P_{j}}  = \overline{P_{j}},
\]
since $\overline{P_{\{j\}}}\subset \overline{F_j}$.

Now we prove (3).  Since $\overline x \in G$, we have $r([k]\setminus\{j\},\{j\}) = n+1$, and hence
    \begin{equation}\label{eqn:largeangles}
    c \le |\det( B_{[k]\setminus\{j\}},A_{j})|.
    \end{equation}
    Let $u_1,\dots,u_{n+1}$ be the columns of $( B_{[k]\setminus\{j\}},A_j)$.
    By exterior algebra,
    \[
    |\det( B_{[k]\setminus\{j\}},A_j)| = |u_1\wedge \dots\wedge u_{n+1}|.
    \]

    Let $w_1,\dots,w_{n+1-n_j}$ be the columns of $B_{[k]\setminus\{j\}}$, and let $v_1,\dots, v_{n_j+1}$ be the columns of $A_{j}$, so $u_1,\dots, u_{n+1} = w_1,\dots,w_{n+1-n_j},v_1,\dots, v_{n_{j+1}}$. 

    By standard facts from exterior algebra (see, e.g., \cite[Corollary to Theorem 7.1]{Afriat_1957} and \cite[Theorem 2.1]{Ispen}),
    \[
    | u_1\wedge\dots \wedge u_{n+1} | =  |w_1 \wedge \dots \wedge w_{n+1-n_j}|\cdot| v_1\wedge \dots \wedge v_{n_j+1}|\cdot|\sin\Theta|,
    \]
    where $\Theta>0$ is the smallest principal angle between the column space of $A_{j}$, namely $\overline F_j$, and the column space of $B_{[k]\setminus\{j\}}$, namely $\overline{P_{[k] \setminus \{j\}}}$.   By exterior algebra again, and since the columns of $A_j$ are orthonormal,
    \[
    |v_1\wedge \dots \wedge v_{n_j+1}|^2 = \det(A_{j}^tA_{j}) = 1.
    \]
    For $S\subset[n+1]$ with $|S| = n+1-n_j$, denote by $(B_{[k]\setminus \{j\}})_S$ the minor matrix obtained by selecting rows from $B_{[k]\setminus \{j\}}$ indexed by $S$. By the Cauchy--Binet formula,
    \begin{align*}
    |w_1 \wedge \dots \wedge w_{n+1 - n_j}|^2 &= \det\big((B_{[k]\setminus \{j\}})^t(B_{[k]\setminus \{j\}})\big)\\
    &=\sum_{S\subset [n+1],|S|=n+1-n_j}\det\big((B_{[k]\setminus \{j\}})_S\big)^2\\
    &\lesssim_{n} M_{n+1-n_j}(B_{[k]\setminus \{j\}})^2 \lesssim_n 1.
    \end{align*}
    So, overall, by \eqref{eqn:largeangles} we have 
    \[
    c \lesssim_{n} |\sin\Theta|.
    \]
    This proves (3), and finishes the proof.
\end{proof}

By $c$-stability and rank conditions $r([k],\emptyset) = n$ and $r([k]\setminus\{j\},\{j\}) = n+1$, we also have a $\delta$-thickened version of Proposition \ref{prop:nice1}.
\begin{prop}\label{prop:key-contain}
    There exists an absolute constant $C > 0$ such that for each $\overline x \in G$, and every $\delta > 0$ we have
    \[
    \overline{P_{[k]\setminus\{j\}}(\overline x)}(\delta)
    \cap\overline F_j \subset  (\overline{P_{[k]\setminus\{j\}}(\overline x)}
    \cap \overline F_j)(Cc^{-1}\delta) = B(0,Cc^{-1}\delta).
    \]
\end{prop}
\begin{proof}
   Let $z \in \overline{P_{[k] \setminus \{j\}}}(\delta)\cap\overline{F_j}$. Since $\overline{P_{[k] \setminus \{j\}}} \cap \overline F_j = \{0\}$ by Proposition \ref{prop:nice1}, it suffices to assume $|z| > Cc^{-1}\delta$. 
    Let  $w\in  \overline{P_{[k] \setminus \{j\}}}$ be the closest point in $ \overline{P_{[k] \setminus \{j\}}}$ to $z$; in particular, $|w|\ge \frac{1}{2}Cc^{-1}\delta>0$ as long as $C$ is sufficiently large, $w-z$ is orthogonal to $w$, and $|w-z|<\delta$.  Consider the right triangle $\triangle 0wz$, whose angle at $0$ is $\ge \Theta \gtrsim c$. By considering the right triangle $\triangle 0wz$, 
    \[
    c\lesssim\sin(\angle z,w) = \frac{\delta}{|z|}.
    \]
    This finishes the proof.
\end{proof}

\begin{prop}\label{prop:delta-nice}
There exists a constant $C>0$ independent of $\overline x\in G$ such that for every $j\in[k]$,
    \[
    \overline{P_{[k]}(\overline x)}
    \cap \overline{F_j}  \subseteq 
    \overline{P_{j}(\overline x)}
    (Cc^{-1}\delta).
    \]
\end{prop}

\begin{proof}
By $\overline{P_{[k]}}=\overline{P_{j}} \oplus \overline{P_{[k] \setminus \{j\}}}$, we have the immediate containment
\[
 \overline{P_{[k]}}(\delta)\cap \overline{F_j} \subset   \overline{P_j}(\delta) + \overline{P_{[k] \setminus \{j\}}}(\delta)\cap \overline{F_j}.
\]
as well as $\overline{P_j}(\delta) \cap \overline{F_j} \subset \overline{P_j}(\delta)$. By Proposition \ref{prop:key-contain}, we also have $\overline{P_{[k] \setminus\{j\}}}(\delta) \cap \overline{F_j} \subset B(0,Cc^{-1}\delta)$, so the claim is proved since $\overline{P_j}(\delta)+B(0,Cc^{-1}\delta)\subset \overline{P_{j}}(2Cc^{-1}\delta)$.
\end{proof}

\begin{prop}
    The measures $\mu_{j,i}$ have $(\sigma,(Cc^{-1})^\sigma\max_jK_j,1-k\varepsilon)$-thin $(n-1)$-planes with respect to the graph $G$.
\end{prop}

\begin{proof}
    For $\overline x$ fixed, and $I\subset[k]$, let $P_{I} = P_I(\overline x)$.
    
    By Proposition \ref{prop:delta-nice} we have the upshot, valid uniformly in $\overline x\in G$:
    \[
    P_{[k]}(\delta)\cap F_j   \subseteq P_j(Cc^{-1}\delta).
    \]
    Hence, by the thin-planes property of $\mu_{j,i}$ on $F_j$, 
    \[
    \mu_{j,i}(P_{[k]}(\delta)) \le \mu_{j,i}(P_j(Cc^{-1}\delta)) \le K_j(Cc^{-1})^\sigma \delta^{\sigma}.
    \]
    The $\prod_{j,i}\mu_{j,i}$ measure of $G$ being at least $1-k\varepsilon$ follows since the $\prod_{i\in I_j}\mu_{j,i}$-measure of each $G_j$ is at least $1-\varepsilon$.
\end{proof}

By replacing $\varepsilon$ with $\varepsilon/k$ and defining $$K:=K(\varepsilon, s, \sigma, C, c, w) = (Cc^{-1})^\sigma\max_jK_j(\varepsilon/k, s, \sigma, C, c, w)$$ we get the desired $(\sigma, K, 1-\varepsilon)$-thin planes graph $G$ for $(\mu_{j, i}:~ j\in [k], i\in [n_j])$. This concludes the proof of the special case $\sum n_j = n$. 

\subsection{\texorpdfstring{General case $p\ge 1$}{General case p≥1}} \label{sec:general-case}

\begin{prop}[cf. \cite{do2018extending} (5.5)]\label{prop:induction-rank-modified}
     Let $F_1,\dots, F_k \subset \R^n$ be a minimal collection of flats with $\sum n_j = n$. For any disjoint $I \subset [k-1]$, $J \subset [k]\setminus I$:
    \[
    r(I, J) = \begin{cases}
    n_I, ~ \text{ if }J=\emptyset,\\
    \ge n_{I \cup J}+1, ~\text{ if }J\neq \emptyset, J\neq[k]\setminus I, \\
    n+1, \text{ if }J=[k]\setminus I
    \end{cases}
    \]
\end{prop}

\begin{proof}
    The proof is exactly the same as Proposition \ref{prop:induction-rank}, except for the base case $I=\emptyset$ and $J=[k]$ we get $r(\emptyset, [k]) = n+1$.
\end{proof}

Let $G_j$ be a $(\sigma, K_j, 1-\varepsilon)$-thin $(n_j-1)$-planes graph for $(\mu_{j,i}:~i\in [n_j])$ and define $G = \prod_{j=1}^{k-1} G_j$. Note that $G$ has measure at least $1-(k-1)\varepsilon$. 
For the last flat $F_k$, we are going to select a subset of $n_k-p$ measures $\mu_{k, i}$ to complete the thin-planes graph $G$. 
For $\overline{x} \in G$ we define $Q(\overline{x}) = P_{ [k-1] }(\overline{x}) \cap F_k$ and $Q_j(\overline{x}) = (P_{[k-1]\setminus \{j\},\{j\}}(\overline{x})) \cap F_k$ for $j\in [k-1]$. 

\begin{prop}\label{prop:Q-is-a-flat}
    Given $\overline x \in G$, let
    \[
    Q(\overline x) = P_{[k-1]}(\overline x)\cap F_k,\quad Q_j(\overline x)= \langle P_{[k-1] \setminus \{j\}}(\overline{x}), F_{j}\rangle \cap F_k.
    \]Then
    \begin{enumerate}
        \item $Q(\overline x)$ is a $p-1$-dimensional flat, and
        \item $Q_j(\overline x)$ is a $p$-dimensional flat.
    \end{enumerate}
\end{prop}
\begin{proof}
    Part (1) follows by considering the matrix $(B_{[k-1]}(\overline x),A_k)$. Given $r([k-1],\{k\}) = n+1$, it follows that $\dim (\overline{P_{[k-1]}(\overline x)}+\overline {F_k}) = n+1$. Therefore, by the dimension-sum formula, 
    \[
    \dim \overline{Q(\overline x)} = n_{[k-1]}+(n_k+1)-(n+1) = p.
    \]
    It follows that $Q(\overline x)$ is a $p-1$-flat. 
    
    For (2), consider the matrix $(B_{[k-1]\setminus\{j\}}(\overline x),A_j,A_k)$. By $r([k]\setminus\{j,k\},\{j,k\}) = n+1$, we have
    \[
    \dim \overline{Q_j(\overline x)} = (n_{[k-1]}-n_j)+(n_j+1)+(n_k+1)-(n+1) = p+1. \qedhere
    \]
\end{proof}

\begin{setup}\label{Setup:Graph-H-and-Measures-nu}
Set $\mathbf{I} = \{1\}\times [p]$ and $\mathbf I^c$ the complement of $\mathbf I$ in $\bigcup_{j=1}^{k-1}\{j\} \times [n_j]$. For $\overline{x}_{\mathbf{I}^c} \in \prod_{(j, i) \in \mathbf{I}^c} \supp\mu_{j,i}$ and $x\in \bigcup_{i=1}^p\supp\mu_{1,i}$ let us define a point $y(x, \overline{x}_{\mathbf{I}^c}) \in F_k$ as the unique point of intersection 
\[
\langle x , \overline{x}_{\mathbf{I}^c}\rangle \cap F_k = \{y\}.
\]
This intersection indeed consists of a single point provided that $x \in \bigcup_{i=1}^p \supp\mu_{1, i}$---this follows from Proposition \ref{prop:Q-is-a-flat}. Denote $y_i = y(x_{1,i},\overline x_{\mathbf I^c})$.

Let $G|_{{\overline{x}_{\mathbf{I}^c}}} \subset \prod_{i=1}^p \supp \mu_{1,i}$ be the set of tuples $\overline{x}_{\mathbf{I}} = (x_{1, 1}, \ldots, x_{1, p})$ so that $(\overline{x}_{\mathbf{I}}, \overline{x}_{\mathbf{I}^c}) \in G$. For a fixed $\overline{x}_{\mathbf{I}^c}$ we can thus define measures
\[
\nu_i^{\overline{x}_{\mathbf{I}^c}} = y( \mu_{1, i}, \overline{x}_{\mathbf{I}^c} )
\]
and the graph
\[
H^{\overline{x}_{\mathbf{I}^c}} = \{ (y(x_{1,1}, \overline{x}_{\mathbf{I}^c}), \ldots, y(x_{1,p}, \overline{x}_{\mathbf{I}^c})):  (x_{1, 1}, \ldots, x_{1, p}) \in G|_{\overline{x}_{\mathbf{I}^c}} \} \subset \prod_{i=1}^p \supp \nu_i^{\overline{x}_{\mathbf{I}^c}}.
\]
\end{setup}

We claim that $H^{\overline{x}_{\mathbf{I}^c}}$ witnesses thin $(p-1)$-planes for $\{\nu_i^{\overline{x}_{\mathbf{I}^c}}\}_{i=1}^p$.

\begin{lemma}[Key Lemma]\label{lem:key-lemma}
For a typical choice of $\overline{x}_{\mathbf{I}^c} \in \prod_{(j, i) \in \mathbf{I}^c} \supp\mu_{j,i}$ there exists a subset $E_i(\overline{x}_{\mathbf{I}^c})$ with $\nu_i^{\overline{x}_{\mathbf{I}^c}}(E_i(\overline{x}_{\mathbf{I}^c})) \ge 1-O(\varepsilon)$ and such that
the measures $\nu_i^{\overline{x}_{\mathbf{I}^c}}|_{E_i(\overline{x}_{\mathbf{I}^c})}$ are $(\tilde C, \sigma)$-Frostman and the graph $H^{\overline{x}_{\mathbf{I}^c}}$ is $(\sigma, \tilde K, 1-O(\varepsilon))$-thin $(p-1)$-planes graph for $(\nu_i^{\overline{x}_{\mathbf{I}^c}}, ~i=1, \ldots, p)$.  
\end{lemma}

We prove the Key Lemma in Section \ref{sec:proof-of-key-lemma}. Using Lemma \ref{lem:key-lemma} we can now finish the proof as follows. 

We first show that spanning thin planes implies having an \textbf{NC} collection of flats.

\begin{prop} \label{prop:thin-implies-NC}
    Let $\mu_1, \ldots, \mu_n$ be measures on $\R^n$ and suppose that $G \subset \prod_{i=1}^n \supp\mu_i$ is a $(\sigma, K, c)$-thin hyperplanes graph for $(\mu_1, \ldots, \mu_n)$ for some $\sigma, K, c >0$. Then the flats $W_j = \aff\supp\mu_j$, $j=1, \ldots, n$, are {\bf NC} in $\R^n$.
\end{prop}

\begin{proof}
    Recall that ${V}_{\overline x} = \aff(x_i : x_i \in \overline x)$. Furthermore let $\overline{V_{\overline{x}}}$ be the linearization of $V_{\overline x}$.

    Suppose that $\bigcup W_j \subset \bigcup U_t$ for some collection of flats $U_t$. For $\overline{x} =(x_1, \ldots, x_n) \in G$ consider the corresponding hyperplane $V_{\overline{x}}$. Then the thin-planes property implies that $W_j \not\subset V_{\overline{x}}$ for every $j$ (otherwise $\mu_j(V_{\overline{x}}(\delta))=1$ for all $\delta>0$). So we also get $U_t \not\subset V_{\overline{x}}$ for all $t$. So $\dim V_{\overline{x}}\cap U_t \le \dim U_t-1$. Equivalently, $\dim (\overline{V_{\overline{x}}} \cap \overline{U_t}) \le \dim U_t$ holds. 
    We conclude that 
    \[
    n = \dim \overline{V_{\overline{x}}} \le \sum_t \dim (\overline{V_{\overline{x}}} \cap \overline{U_t}) \le \sum_t \dim U_t. \qedhere
    \]
\end{proof}

Since by assumption $G_k$ is a thin planes graph for the measures $(\mu_{t, i}, ~t\in I_k, i\in [d_t])$, we conclude that the collection of flats $(V_{t}, ~ t \in I_k)$ is ${\bf NC}$ inside $F_k$. 

Let $Z \subset F_k$ be a generic subspace of dimension $n_k-p$. For $\overline{x} \in G$ and $v \in F_k \setminus Q(\overline {x})$, there exists a unique $y(v) \in Z$ such that 
\[
\aff(v, Q(\overline x))\cap Z = \{y(v)\}.
\]
This gives us the \emph{join-meet (radial)} projection map $\pi_{Q(\overline{x})}^Z: F_k \setminus Q(\overline{x}) \to Z$ defined by $\pi_{Q(\overline x)}^Z(v) = y(v)$.

\begin{prop}\label{prop:projection-is-still-NC}
    For all but finitely many $\overline{x} \in G$, the collection of flats $\pi_{Q(\overline{x})}^Z(V_t)$, $t\in I_k$ is {\bf NC} inside $Z$.
\end{prop}

\begin{proof}
    Observe that $Q(\overline{x}) = \langle \overline{y}_{[p]}\rangle$ for $\overline{y}_{[p]} \in H^{\overline{x}_{\mathbf {I}^c}}$ and by Proposition \ref{prop:thin-implies-NC}, $(V_t, t\in I_k)$ is an $\mathbf{NC}$ collection of flats. 

    Pick generic $n_k-1$-flats $Z_1,\dots,Z_p\subset F_k$ such that $\bigcap_{i=1}^p Z_i = Z$.
    Then, the join-meet projections satisfy
    \[
    \pi_{\langle \overline y_{[p]}\rangle}^Z = \pi_{y_p}^{Z_p}\circ\dots\circ \pi_{y_1}^{Z_1}.
    \]
    It follows from applying Lemma \ref{lem:projecting-NC-flats} $p$ times that the projected flats are {\bf NC}. 
\end{proof}

By Lemma \ref{lem:projecting-stable}, the measures $\pi_{Q(\overline{x})}^Z\mu_{t, i}$, $t\in I_k$, $i \in [\dim V_j]$ are in $c^{O(1)}$-stable position with respect to flats $\pi_{Q(\overline{x})}^Z (V_t) \subset Z$. By Proposition \ref{lemma:projection-of-irreducible}, as long as $r>0$ is sufficiently small, the measures $\pi_{Q(\overline{x})}^Z\mu_{t, i}|_{F_{k} \setminus Q(\overline{x})(r)}$, $t\in I_k$, $i \in [\dim V_j]$, are $(c^{O(1)}w, 2 \tau)$-irreducible inside flats $\pi_{Q(\overline{x})}^Z (V_t)$ (as long as the projection is at least one-dimensional). 

Finally, by the Key Lemma (Lemma \ref{lem:key-lemma}) and Ren's discretized radial projections theorem, for typical $\overline{x}$, we can restrict $\mu_{t, i}$ to a set $S_{t, i}(\overline{x})$ of measure $\ge 1-O(\varepsilon)$ in such a way that $\pi_{Q(\overline{x})}^Z(\mu_{t, i}|_{S_{t, i}(\overline{x})})$ is a $(\tilde C, \sigma)$-Frostman measure on $U$. 

So we are in position to apply Theorem \ref{thm:general-case} to the collection of flats $\pi_{Q(\overline{x})}^Z (V_t) \subset Z$ and measures $\pi_{Q(\overline{x})}^Z(\mu_{t, i}|_{S_{t, i}(\overline{x})})$, $t\in I_k$, $i \in [\dim V_j]$. Note that $\dim Z = n_k-p < n$, so by induction we may assume we already proved Theorem \ref{thm:general-case} in dimension $n_k-p$. Hence we conclude that there are some numbers $d'_t \le \dim \pi_{Q(\overline{x})}^Z (V_t)$ so that for typical $\overline{x}$, the measures
\begin{equation}\label{eq:thin-planes-restricted-to-S}
(\pi^Z_{Q(\overline{x})}(\mu_{t, i}|_{S_{t, i}(\overline{x})}), ~~ t\in I_k, ~i\in [d'_t])    
\end{equation}
form $(\sigma', K', 1-\varepsilon')$-thin hyperplanes in $Z$. Note that we can use the stable position property to ensure that the indices $d'_t$ are the same for all choices of $\overline{x}$ (if some indices work for one $\overline{x}$ then they work for almost all by pigeonholing). 

Iterating Lemma \ref{lem:swapping} with $\nu =  \pi_{Q(\overline{x})}^Z\mu_{t, i}$ for each $(t,i)$, we can upgrade the thin planes graph for (\ref{eq:thin-planes-restricted-to-S}), to a thin planes graph of the unrestricted measures 
\begin{equation}\label{eq:thin-planes-no-S}
(\pi_{Q(\overline{x})}^Z(\mu_{t, i}), ~~ t\in I_k, ~i\in [d'_t]).
\end{equation}
That is, we get that the measures (\ref{eq:thin-planes-no-S}) have $(\sigma', K'', 1-\varepsilon')$-thin planes for some $K''$ depending on $K'$ and the remaining parameters. 
Let $G^{(\overline{x})} \subset \prod_{t\in I_k, i \in [d'_t]} \supp \mu_{t, i}$ be the preimage of the corresponding thin-planes graph under the projection map. More precisely, let $\mathbf J = \bigcup_{t\in I_k}\{t\}\times[d_t']$, and we define $G^{(\overline{x})}$ to be the set of tuples $\overline{x_{\mathbf J}} = (x_{t,i}:(t,i)\in\mathbf J) \in \prod_{(t,i)\in\mathbf J} \supp \mu_{t, i}$ such that 
\begin{equation}\label{eqn:mu-ti-havethinplanes}
\pi_{Q(\overline{x})}^Z(\mu_{t, i}) (\langle \pi_{Q(\overline{x})}^Z(x_{t, i}), ~(t,i)\in\mathbf J\rangle(\delta)) \le K'' \delta^{\sigma'}
\end{equation}
holds for all dyadic $\delta$. Then it follows that this graph has density $\ge 1-\varepsilon'$ for any $\sigma'< \sigma$ and sufficiently large $K''$. 

With this definition it is clear that the following graph is Borel:
\[
\overline{G} = \{ (\overline{x},\overline{x_{\mathbf J}}): ~ \overline{x} \in G', ~\overline{x_{\mathbf J}} \in G^{(\overline{x})} \}
\]
where $G' \subset G$ is the set of $\overline{x}$ for which the typical conditions above hold. 

\begin{prop}\label{prop:separate-from-Qj}
    For every $(\overline{x}, \overline{x_{\mathbf J}})\in \overline{G}$ and every $j\in[k-1]$, we have
    \[
    \angle (Q_j(\overline{x}), \langle Q(\overline x), \overline{x_{\mathbf J}}\rangle) \gtrsim c^{O(1)}.
    \]
\end{prop}

\begin{proof}
    Let $\tilde \nu^{(\overline{x})}$ be the uniform probability measure on the points $\pi_{Q(\overline{x})}^Z(Q_j(\overline{x})) \in Z$ for $j=1, \ldots, k-1$ (note that $Q(\overline{x})\subset Q_j(\overline{x})$ and these are $p-1$ and $p$-flats, respectively). So by Lemma \ref{lem:swapping} we can find a subgraph $G'^{(\overline{x})} \subset G^{(\overline{x})}$ so that for any $\overline{x}_{\mathbf J} \in G'^{(\overline{x})}$ the hyperplane $\langle \pi_{Q(\overline{x})}^Z(\overline{x_{\mathbf J}})\rangle \subset Z$ is $\gtrsim_{\varepsilon,K', n} 1$ separated from the support of $\tilde \nu^{(\overline{x})}$. So the graph defined as
    \[
    \overline{G}' = \{(\overline{x}, \overline{x_{\mathbf J}}):~ \overline{x_{\mathbf J}} \in G'^{(\overline{x})}\}
    \]
    has the property that so for every $(\overline{x}, \overline{x_{\mathbf J}})\in \overline{G}'$ and every $j\in[k-1]$, we have
    \[
    \angle (Q_j(\overline{x}), \langle Q(\overline x), \overline{x_{\mathbf J}}\rangle) \gtrsim_{\varepsilon,K'', n} 1.
    \]
    since our measures are in stable position, the above condition implies a stronger property:
    \[
    \angle (Q_j(\overline{x}), \langle Q(\overline x), \overline{x_{\mathbf J}}\rangle) \gtrsim c^{O(1)}
    \]
    since this can be written as a statement about ranks and determinants of the matrices $B_*(\overline{x})$. So having the angle separation hold for a single choice of the tuple ensures that it holds for all tuples.
\end{proof}

We claim that $\overline{G}$ witnesses $\min\{\sigma,\sigma'\}$-thin hyperplanes for $$(\mu_{j,i}: (j,i) \in (\bigcup_{j\in[k-1]}\{j\}\times[n_j])\cup\mathbf J).$$
Now we check the thin-planes property. Let $(\overline{x}, \overline{x_{\mathbf J}}) \in \overline{G}$ be arbitrary and let $V_{\overline{x}, \overline{x_{\mathbf J}}}$ be the affine span of these vectors in $\R^n$. First we check that $\mu_{t, i}(V_{\overline{x}, \overline{x_{\mathbf J}}}(\delta))$ is small enough. First, notice that 
\[
V_{\overline{x}, \overline{x_{\mathbf J}}}(\delta) \cap F_k \subset \langle Q(\overline{x}), \overline{x_{\mathbf J}}\rangle(C\delta)
\]
this is because $\angle (V_{\overline{x}}, F_k) \gtrsim c^{O(1)}$ using the $c$-stable condition. 
By using the $c$-stable condition again, we see that $\mu_{t, i}$ must be $\sim c$-separated from the space $Q(\overline{x})$ (indeed a typical $x \in \operatorname{supp}\mu_{t, i}$ does not lie in $Q(\overline{x})$ and thus, the distance from any $x' \in \operatorname{supp}\mu_{t, i}$ to $Q(\overline{x})$ has to be lower bounded by $\sim c$ since this distance can be expressed in terms of functions $M_r(\cdot)$). Thus, by \eqref{eqn:mu-ti-havethinplanes} we get
\[
\mu_{t, i}(\langle Q(\overline{x}), \overline{x_{\mathbf J}}\rangle(C\delta)) \le \pi_{Q(\overline{x})}^Z\mu_{t,i}(\langle \pi_{Q(\overline{x})}^Z(\overline{x_{\mathbf J}})\rangle (Cc^{-O(1)} \delta) ) \le K'' (Cc^{-O(1)} \delta)^{\sigma'},
\]
by the thin-planes property of $\pi_{Q(\overline{x})}^ZG^{(\overline{x})}$. So we conclude that $\mu_{t, i}(V_{\overline{x}, \overline{x_{\mathbf J}}}(\delta)) \le K'' (Cc^{-O(1)} \delta)^{\sigma'}$.

Now we estimate $\mu_{j, i}(V_{\overline{x},\overline{x_{\mathbf J}}}(\delta))$ for some $j \in[k-1]$ and $i \in [n_j]$. By Proposition \ref{prop:separate-from-Qj}, we have $V_{\overline x, \overline x_{\mathbf J}} + Q_j(\overline x) = \R^n$, and thus by the dimension-sum formula we have 
\[
\dim V_{\overline{x}, \overline{x_{\mathbf J}}} \cap Q_j(\overline{x}) = p-1,
\]
where recall that $Q_j(\overline{x}) = \langle V_{\overline{x}}, F_j \rangle \cap F_k$.
Furthermore, since $Q(\overline x) \subset Q_j(\overline{x})$ is a $p-1$ plane, it follows that
\[
V_{\overline{x}, \overline{x_{\mathbf J}}} \cap Q_j(\overline{x}) = Q(\overline{x}), 
\]
 So we obtain, 
\[
F_k \cap \langle \overline{x}, \overline{x_{\mathbf J}}\rangle \cap \langle \overline{x}, F_j\rangle= F_k \cap \langle \overline{x}\rangle = Q(\overline{x}).
\]
On the other hand, if $F_j \subset \langle \overline{x}, \overline{x_{\mathbf J}}\rangle$ then we obtain
\[
Q_j(\overline{x})=F_k \cap \langle \overline{x}, F_j\rangle \subset F_k \cap \langle \overline{x}, \overline{x_{\mathbf J}}\rangle \cap \langle \overline{x}, F_j\rangle = Q(\overline x)
\]
a contradiction. So we conclude that $F_j \not\subset \langle \overline{x}, \overline{x_{\mathbf J}}\rangle$ and since we already know that $\langle \overline{x}\rangle\cap F_j = \langle \overline{x}_j\rangle$ is  a hyperplane on $F_j$, we get that $F_j \cap \langle \overline{x}, \overline{x_{\mathbf J}}\rangle = \langle \overline{x}_j\rangle$. Furthermore, using the $c$-stable position we immediately conclude that $\angle (F_j, \langle \overline{x}, \overline{x_{\mathbf J}}\rangle ) \gtrsim c$ and so 
\[
\mu_{j, i}(V_{\overline{x}, \overline{x_{\mathbf J}}}(\delta)) \le \mu_{j, i}(\langle \overline{x}_j\rangle(C'c^{-1}\delta)) \le K_j (C'c^{-1}\delta)^{\sigma}. 
\]
Thus, the constructed graph $\overline{G}$ is a $(\min(\sigma, \sigma'), C c^{-O(1)}\max \{K'', K_j : j\in [k-1]\}, 1-O(\varepsilon))$-thin planes graph.

\subsection{Proof of the key lemma} \label{sec:proof-of-key-lemma}
We now prove the Key Lemma which we restate here for the convenience of the reader.

\begin{lemma}[Key Lemma]
For a typical choice 
of $\overline{x}_{\mathbf{I}^c} \in \prod_{(j, i) \in \mathbf{I}^c} \supp\mu_{j,i}$ there exists a subset $E_i(\overline{x}_{\mathbf{I}^c})$ with $\nu_i^{\overline{x}_{\mathbf{I}^c}}(E_i(\overline{x}_{\mathbf{I}^c})) \ge 1-O(\varepsilon)$ and such that
the measures $\nu_i^{\overline{x}_{\mathbf{I}^c}}|_{E_i(\overline{x}_{\mathbf{I}^c})}$ are $(\tilde C, \sigma)$-Frostman and the graph $H^{\overline{x}_{\mathbf{I}^c}}$ is $(\sigma, \tilde K, 1-O(\varepsilon))$-thin $(p-1)$-planes graph for $(\nu_i^{\overline{x}_{\mathbf{I}^c}}, ~i=1, \ldots, p)$.  
\end{lemma}

Note that the first statement about the restricted measures being $(\tilde C, \sigma)$-Frostman follows from the thin-planes property---say, if measures $\mu_1, \ldots, \mu_n$ have $(\sigma, K, 1-\varepsilon)$-thin planes graph $G$, then we can define 
\[
E_i = \{x: (\prod_{i'\neq i} \mu_{i'})(G|_{x_i=x}) \ge 1/2\}
\]
and notice that $\mu_i|_{E_i}$ is $(\sigma, 2\tilde K)$-Frostman and $\mu_i(E_i) \ge 1-O(\varepsilon)$. 

Now we show that the graph $H^{\overline{x}_{\mathbf{I}^c}}$ has the thin-planes property by transferring the thin-planes property of the graph $G|_{\overline{x}_{\mathbf{I}^c}} \subset \prod_{i=1}^p \supp\mu_{1,i}$. This will hold by the following construction.

\subsubsection{Choosing a generic flat} \label{sec:choosing-a-generic-flat} For $\overline x \in G$, define the $n_1-p-1$-flat
\[
C = C(\overline x) = C(\overline x_{\mathbf I^c\cap(\{1\}\times [n_1])}) := \langle \overline x_{\mathbf I^c\cap(\{1\}\times [n_1])}\rangle \subset F_1,
\]
and
\[
\mathcal C = \{C(\overline x):\overline x \in G\},
\]
and for $i = 1,\dots,p$ consider the collection of all joins formed between $\mu_{1,i}$ and $C(\overline x)$:
\[
\mathcal J_{i} := \{\langle x_{1,i}, C(\overline x_{\mathbf I^c})\rangle : \overline x \in G, x_{1,i}\in \supp\mu_{1,i}\}.
\]
For each $i$, we claim that $\mathcal J_i\subset B_{\mathcal A(F_1,n_1-p)}(P_i,\eta)$ for some $P_i \in \mathcal A(F_1,n_1-p)$ and some $0 < \eta.$

By $c$-stable position, for all $i\neq i',$
\[
\angle P_i, P_{i'} \gtrsim c.
\]
So if we choose $\eta>0$ sufficiently small, there exists $P_0\in \mathcal A(F_1,n_1-p)$ with $\mathrm{dist}(P_0,P_i)\ge \eta$ for every $i$. Pick a translate $U$ of $P_0^\perp$ in $F_1$ (where we take orthogonal complements within $\mathrm{dir}(F_1)$), $c$-separated from $\bigcup \mathcal C$.

Notice that with $C:= \langle\overline{{x}}_{\mathbf I^c\cap(\{1\}\times [n_1])}\rangle$, it follows that for all $v\in F_1 \setminus C$, there exists a unique $\pi(v) \in U$ such that
\[
\aff(v, C) \cap U = \{\pi(v)\}.
\]
We denote this radial projection
\[
\pi=\pi_{C(\overline x)}^U : F_1\setminus C\to U
\]
Note that having fixed $\overline{x}_{\mathbf I^c\cap(\{1\}\times [n_1])}$, the intersection $Q(\overline{x}) = P_{[k-1]}(\overline{x}) \cap F_k$ is completely determined by the hyperplane $W = \pi (\langle \overline{x}_{\mathbf I}\rangle) = \langle \pi(x_{1,i}) : i \in [p]\rangle $. We choose affine coordinates $\Phi_{\overline x}$ defined by $(u,t,w)\in\R^p\times\R\times\R^{n_1-p-1}$ on $F_1$ so that $U = \{(u,0,0):u\in\R^p\}$, $C(\overline x) = \{(0,1,w):w\in\R^{n_1-p-1}\}$.

In the $\Phi_{\overline x}$ coordinates, the map $\pi\colon F_1\setminus C(\overline x)\to U$ becomes
\[
\pi(u,t,w) = (\frac{u}{1-t},0,0).
\]
By $c$-stable position, the coordinate systems $\Phi_{\overline x}$ all satisfy
\begin{equation}\label{eq:coord-lipschitz}
\mathrm{Lip}(\Phi_{\overline x})\lesssim c^{-O(1)}.
\end{equation}

\begin{prop}
Consider the flats $E = \langle \overline x_{\mathbf I^c}\rangle$, $J = \langle E, F_1\rangle$, and $Q_1(\overline  x) = \langle P_{[k-1] \setminus \{1\}}(\overline{x}), F_{1}\rangle \cap F_k$. 

    Each of the following hold:
    \begin{enumerate}
    \item $\dim E = n_{[k-1]}-p-1$
    \item $J  
=\langle P_{[k-1]\setminus\{1\}}(\overline x),F_1\rangle$
    \item $\dim J = n_{[k-1]}$
    \item $Q_1(\overline x) =  J\cap F_k$
    \end{enumerate}
\end{prop}
\begin{proof}
    (1) This follows from $r([k-1],\emptyset) = n_{[k-1]}$ and noting that augmenting a matrix by $p$ columns at most increases the rank by $p$. Therefore we have
    \begin{align*}
    c\lesssim M_{n_{[k-1]}}(\overline x_{1,1},\dots,\overline x_{1,p},\overline x_{\mathbf I^c}) \le M_{n_{[k-1]}-p}(\overline x_{\mathbf I^c}),
    \end{align*}
    proving (1).
    
    (2) This follows by considering the images of the matrices $(B_{[2,k]}(\overline x),A_1)$ and $(B_{\mathbf I^c}(\overline x),A_1)$ and the dimension-sum formula.
    
    (3) We have $\dim\langle P_{[k-1]\setminus\{1\}}(\overline x),F_1\rangle = r([k-1] \setminus \{1\},\{1\})-1\ge (n_{[k-1]}+1)-1 = n_{[k-1]}$, by Proposition \ref{prop:induction-rank}. Equality is obtained by noting that there are exactly $n_{[k-1]}$ column vectors in the matrix $(B_{[k-1] \setminus \{1\}} (\overline x), A_1)$

    (4) Immediate from (2) and the definition of $J$.
\end{proof}

If $W$ is a $p-1$ plane on $U$ defined by $a^\top u = b$ with $a \ne 0$, then the unique hyperplane $H_W\subset F_1$ containing $C$ satisfying $\pi(H_W\setminus C) = W$ is defined by the equation
\[
a^\top u + b(t-1) = 0.
\]

\begin{lemma}[Dimension and hyperplane lifts]\label{lem:SW-hyperplane}
For any hyperplane $W\subset U$, the lift $H_W\subset F_1$ given by
$a^\top u+b(t-1)=0$ (in the $(u,t,w)$-chart) contains $C$ and projects to $W$.
Let $E=\langle \overline x_{\mathbf I^c}\rangle$, $J=\aff(E,F_1)$.
Then
\[
S_W\ :=\ \aff(E,H_W)\ \subset\ J
\]
is a hyperplane in $J$.
\end{lemma}

\begin{proof}
The verification that $H_W$ contains $C$ and $\pi(H_W\setminus C)=W$ is by direct
substitution. For dimensions, $\dim E=n_{[k-1]}-(p+1)$ and
$\dim H_W=n_1-1$, while $E\cap H_W$ contains $C$ of dimension $n_1-p-1$. Note that $\angle E, F_1 \gtrsim c$, and therefore,
\[
\dim(E\cap F_1) = \dim E +\dim F_1 -n_{[k-1]} = n_{[k-1]}-p-1 + n_1 - n_{[k-1]} = n_1 - p - 1.
\]
Therefore, $E\cap H_W = C$, and hence
\begin{align*}
\dim S_W&=\dim E+\dim H_W-\dim C
\\
&=n_{[k-1]}-(p+1)+n_1-1-(n_1-p-1)\\
&=n_{[k-1]}-1 \\
&=\dim J-1. \qedhere
\end{align*}
\end{proof}

We now define a map that will allow us to deduce that $\nu_i^{\overline{x}_{\mathbf{I}^c}}$ has $\sigma$-thin $(p-1)$-planes from the fact that $\{\mu_{1,i}\}_{i=1}^p$ has $\sigma$-thin $p-1$ planes.

\begin{definition}[The hyperplane map]\label{def:psi}
For $\overline x \in G$, define $\psi = \psi_{\overline x}$ by
\[
\psi\colon \mathcal A(U,p-1) \to \mathcal A(Q_1(\overline{x}), p-1)
\]
via
\[
\psi(W)\ =\ S_W\cap F_k\ \subset\ Q_1(\overline x)=J\cap F_k .
\]
\end{definition}
\begin{prop}
    For each hyperplane $W\subset U$, $\psi(W)$ is a hyperplane of $Q_1(\overline x)$.
\end{prop}
\begin{proof}
This follows from $r([2,k-1],\{1,k\}) = n+1$, since $\angle (S_W,F_k)\ge \angle( J,F_k) \gtrsim M_{n+1}(A_1,B_{[2,k-1]}(\bar x),A_k)\gtrsim c$.
\end{proof}

We will show $\psi$ is an immersion of $p-1$-planes by justifying an appropriate coordinatization is an immersion. In particular, we will show the following.

\begin{lemma}
    Locally, near $W\in \mathcal A(U,p-1)$, there is a system of coordinates $(a,b)$ so that $W = W(a,b)$, and a matrix $M = M_{\overline x}\in \GL_p(\R)$ such that
    \[
    \psi_{\overline x}(a,b) = (M^{-t}a,b)
    \]
    and such that $\|M_{\overline x}\| = O(c^{-O(1)})$ for all $\overline x \in G$.
\end{lemma}

\subsubsection{\texorpdfstring{$\psi$ is an immersion}{psi is an immersion}}

Pick a transversal $(p+1)$–flat $K\subset J$ with $K\cap E=\emptyset$ and
$\dir(J)=\dir(E)\oplus\dir(K)$. Fix $k_0\in K$ and define the affine projection along $E$,
\[
q_K:J\to K,\qquad q_K(x)=k_0+\Pi_K(x-k_0),
\]
where $\Pi_K$ is the linear projection with kernel $\dir(E)$. Set
$\mathbf U:=q_K(U)$ and $\mathbf Q_1:=q_K\bigl(Q_1(\overline x)\bigr)$; then
$q_K|_U:U\to\mathbf U$ and $q_K|_{Q_1(\overline x)}:Q_1(\overline x)\to\mathbf Q_1$ are affine
bijections by genericity: no direction of $U,Q_1$ is parallel to $E$.

Choose $v\in \mathbf U\cap\mathbf Q_1$ and recenter $K$ by $-v$, so $K$ is now a vector
space and $\mathbf U,\mathbf Q_1$ are codimension-1 subspaces. Fix a line $L\subset K$ with
$K=\mathbf U\oplus L=\mathbf Q_1\oplus L$. Let $P_{\mathbf U}$, $P_L$ be the projections for
this direct sum. The restriction $P_{\mathbf U}|_{\mathbf Q_1}:\mathbf Q_1\to\mathbf U$ is a
linear isomorphism with norm bounded by $c^{-O(1)}$; define the \emph{slope map}
\[
A:\mathbf U\to L,\qquad A:=P_L\circ (P_{\mathbf U}|_{\mathbf Q_1})^{-1}.
\]
Fix a basis $(e_1,\dots,e_p)$ of $\mathbf U$ and a generator $z\in L$. There are unique
scalars $s_1,\dots,s_p$ with $A(e_i)=s_i z$. Writing $u=\sum u_i e_i$ and
$s^\top:=(s_1,\dots,s_p)$ we have
\[
A(u)=(s^\top u)\,z,\qquad
\mathbf Q_1=\{(u,\tau):\ \tau=s^\top u\},\qquad
T(u):=u+A(u)=(u,s^\top u).
\]
Thus $T:\mathbf U\to\mathbf Q_1$ is the linear projection along the \emph{fixed} direction $L$,
with explicit slope vector $s$.

\begin{definition}[The affine isomorphism $L^\#$]
Define
\[
L^\#\ :=\ (q_K|_{Q_1})^{-1}\circ T\circ (q_K|_U):\ U\to  Q_1(\overline x).
\]
Then $L^\#$ is an affine bijection. In affine coordinates on $U$ and $Q_1(\overline x)$,
write
\[
L^\#(u)=y_0+M_{\overline{x}}\,u,\qquad M = M_{\overline{x}}\in \GL_p(\R).
\]
\end{definition}

Because flats are in $c$-stable position, $\|M_{\overline x}\| \lesssim c^{-O(1)}$ uniformly in $\overline x \in G$.

\begin{lemma}[Graph description of $\psi(W)$]\label{lem:graph}
If $W=\{u:\ a^\top u=b\}\subset U$ with $a\neq 0$, then
\[
q_K\bigl(\psi(W)\bigr)
= T\bigl(q_K(W)\bigr)
= \{(u,s^\top u):\ a^\top u=b\}\ \subset\ \mathbf Q_1.
\]
Consequently, $\psi(W)=L^\#(W)$.
\end{lemma}

\begin{proof}
By construction, $S_W=\aff(E,H_W)$ is a hyperplane in $J$, and $q_K$ collapses $E$-fibers.
Hence $q_K(S_W)$ is the hyperplane in $K$ obtained by projecting $H_W$ along $E$; its
intersection with $\mathbf Q_1=q_K(Q_1)$ is exactly $T(q_K(W))$, the graph over $q_K(W)$
along $L$. Since $q_K|_{Q_1}$ is bijective, $\psi(W)$ is the inverse image of this graph,
i.e.\ $L^\#(W)$.
\end{proof}

\begin{prop}[Parameter map]\label{prop:param-map}
Use coordinates on $Q_1(\overline x)$ given by $y=y_0+M\eta$. Then for
$W(a,b)=\{u:\ a^\top u=b\}$,
\[
\psi\bigl(W(a,b)\bigr)\ =\ \{y:\ (M^{-\top}a)^\top \eta=b\}.
\]
Equivalently, the parameter map on hyperplanes is $(a,b)\mapsto (a',b')=(M^{-\top}a,\ b)$.
\end{prop}

\begin{proof}
From $y=y_0+ M u$ we have $u=M^{-1}\eta$. Substituting into $a^\top u=b$ gives
$a^\top M^{-1}\eta=b$, i.e.\ $(M^{-\top}a)^\top\eta=b$.
\end{proof}

\begin{prop}[$\psi$ is an immersion]\label{prop:psi-immersion}
Equip $\mathcal A(U,p-1)$ and $\mathcal A(Q_1(\overline{x}), p-1)$ with the unnormalized
normal and offset charts $(a,b)$ and $(a',b')$. The differential of $\psi$ at $(a,b)$ is
\[
d\psi_{(a,b)}= \begin{bmatrix} M^{-t} & 0 \\ 0 & 1\end{bmatrix}
\]
which is injective since $M\in\GL_p(\R)$. Therefore $\psi$ is a local diffeomorphism with
\begin{equation} \label{eqn:Lip-Upperbound}
\mathrm{Lip}(\psi) \lesssim c^{-O(1)}.
\end{equation}
\end{prop}

\begin{proof}
Differentiate the linear parameter map in Proposition~\ref{prop:param-map}. Equation \ref{eqn:Lip-Upperbound} follows from \eqref{eq:coord-lipschitz} and noting 
\[
\mathrm{Lip}(\psi) \lesssim \|M\| \lesssim c^{-O(1)}. \qedhere
\]
\end{proof}

\begin{prop}\label{prop:immersion-implies-thin}
    Using the notation as in Setup \ref{Setup:Graph-H-and-Measures-nu}, $H^{\overline{x}_{\mathbf I^c}}$ is a $(\sigma, c^{-O(1)}K_1,1-O(\varepsilon))$-thin $(p-1)$-planes graph for the measures $\nu_1^{\overline{x}_{\mathbf{I}^c}},\dots,\nu_p^{\overline{x}_{\mathbf{I}^c}}$.
\end{prop}

\begin{proof}
    Recall
    \[
    H^{\overline{x}_{\mathbf{I}^c}} = \{ (y(x_{1,1}, \overline{x}_{\mathbf{I}^c}), \ldots, y(x_{1,p}, \overline{x}_{\mathbf{I}^c})):  (x_{1, 1}, \ldots, x_{1, p}) \in G|_{\overline{x}_{\mathbf{I}^c}} \} \subset \prod_{i=1}^p \supp \nu_i^{\overline{x}_{\mathbf{I}^c}}.
    \]
    By taking sections of a thin-planes graph, there is a set $G_0 \subset \prod_{(j,i)\in\mathbf I^c}\mu_{j,i}$ of $\mu_{\mathbf I^c}$-measure at least $1-O(\varepsilon)$, such that for every $\overline x_{\mathbf I^c}\in G_0$,
    \[
    G|_{\overline x_{\mathbf I^c}} \ \text{is a $(\sigma,K,1-O(\varepsilon))$-thin $(p-1)$-planes graph}
    \]
    for the measures $\mu_{1,1},\dots,\mu_{1,p}$. Hence, for $\overline x_{\mathbf I^c}\in G_0$,
    \[
    \nu_1^{\overline x_{\mathbf I^c}} \times \dots \times \nu_{p}^{\overline x_{\mathbf I^c}}(H^{\overline x_{\mathbf I^c}}) \ge 1-O(\varepsilon).
    \]
    Furthermore, let $\pi = \pi_{C(\overline{x})}^U$ and $y_i = y(x_{1,i}, \overline{x}_{\mathbf{I}^c})$ for $i = 1,\dots, p$. 
    For tuple $x_{\mathbf{I}} = (x_{1, 1}, \ldots, x_{1, p}) \in G|_{\overline{x}_{\mathbf{I}^c}}$ and the corresponding tuple $(y_1, \ldots, y_p) \in H^{\overline{x}_{\mathbf{I}^c}}$ we have by definition that
    \[
    \psi(\langle \pi(x_{1, 1}), \ldots, \pi(x_{1, p})\rangle) = \langle y_1, \ldots, y_p\rangle,
    \]
    where $C = \langle x_{1,p+1}, \ldots, x_{1, n_1}\rangle$. 
    By Proposition \ref{prop:psi-immersion}, we have
    \[
    \psi^{-1}\left(B_{\mathcal A(Q_1(\overline x),p-1)}(\langle y_1,\dots,y_p\rangle, \delta)\right) \subset B_{\mathcal A(U,p-1)}(\langle \pi(x_{1,1}),\dots,\pi(x_{1,p})\rangle,c^{-O(1)}\delta).
    \]
    Thus,
    \begin{align*}
    \nu_i^{\overline{x}_{\mathbf{I}^c}}(\langle y_1, \ldots, y_p\rangle (\delta)) &\le \pi(\mu_{1, i})\left( \langle \pi(x_{1, 1}), \ldots, \pi(x_{1, p})\rangle (O(c^{-O(1)})\delta) \right) \\
    &\le \mu_{1,i}(\langle x_{1,1},\dots,x_{1,p}\rangle)(O(c^{-O(1)}\delta))\\
    &\le K_1 (O(c^{-O(1)})\delta)^{\sigma}
    \end{align*}
    where in the second-to-last inequality we use the choice of $U$ from Section \ref{sec:choosing-a-generic-flat}, and in the last we use the thin-planes property of $G|_{\overline x_{\mathbf I^c}}$ for every $\overline x_{\mathbf I^c}\in G_0$. 
\end{proof}

This completes the proof that $H^{\overline{x}_{\mathbf{I}^c}}$ is $(\sigma, K_1 (O(c^{-O(1)}))^{\sigma}, 1-O(\varepsilon))$-thin $(p-1)$-planes graph.

\subsection{Extra lemmas}

\begin{lemma}\label{lem:swapping}
    Let $n'\ge n$. 
    Let $\mu_1, \ldots, \mu_n$ be measures on $B^{n'}(0,1)$ and $G \subset \prod \supp \mu_i$ be a $(\sigma, K, c)$-thin $(n-1)$-planes graph. Let $\nu$ be any probability measure on $\R^{n'}$. Then for every $\varepsilon>0$ there is $K' = K'(\varepsilon, K, n')$ and a graph $G'\subset G$ of $\prod\mu_i$-measure at least $c-\varepsilon$ such that  $\nu(\langle x_1, \ldots, x_n\rangle(\delta)) \le K' \delta^{\sigma-\varepsilon}$ for every $(x_1, \ldots, x_n) \in G'$.
\end{lemma}
In particular, this means that we can replace $\mu_1$ with the measure $\mu_1\cup \nu$ and retain the thin-planes property (with slightly different constants). So if we showed $(\mu_1|_{B_1}, \ldots, \mu_n|_{B_n})$ have thin planes, then so do $(\mu_1, \ldots, \mu_n)$. We use this lemma at the end of the proof to get rid of the restriction to balls.
\begin{proof}[Proof of Lemma \ref{lem:swapping}]
    For each $\delta, K'$ we define
    \[
    G_{\delta, K'} = \{(x_1, \ldots, x_n)\in G:~ \nu(\langle x_1, \ldots, x_n\rangle(\delta)) > K'\delta^{\sigma-\varepsilon}\}.
    \]
    Our goal is to show that $(\prod \mu_i)(G_{\delta, K'})$ is appropriately small. The thin-planes property of $G$ implies that most tuples $(x_1, \ldots, x_n)\in G$ are quantitatively affinely independent. In particular for any $\varepsilon>0$ we can find $\delta_0 >0$ and $G'' \subset G$ so that $(\prod\mu_i)(G'') \ge c-\varepsilon$ and for every $(x_1, \ldots, x_n)\in G''$ we have 
    \[
    \mathrm{dist}(x_j, \langle x_1, \ldots, \widehat{x_j}, \ldots, x_n\rangle) \ge \delta_0
    \]
    (and $\delta_0$ is a function of $\varepsilon$). 
    Fix $\delta \ll \delta_0$ and cover $\R^{n'}$ by dyadic $\delta$-boxes $Q_v$ (where $v \in \R^{n'}$ denotes the center of the box). We claim that for a fixed box $Q_v$,
    \begin{equation}\label{eq:prod-Qv}
    (\prod\mu_i)\left\{ (x_1, \ldots, x_n)\in G'':~ \langle x_1, \ldots, x_n\rangle \cap Q_v \neq \emptyset \right\} \lesssim_{n'} K\delta_0^{-2\sigma}\delta^\sigma.   
    \end{equation}
    Indeed, for any $(x_1, \ldots, x_n)\in G''$ we can find $j$ so that $\mathrm{dist}(Q_v, \langle x_1, \ldots, \widehat{x_j}, \ldots, x_n\rangle) \gtrsim_{n'} \delta_0$ (this uses the definition of $G''$ and some simple geometry). Then we note that for any other tuple $(x_1, \ldots, x_j', \ldots, x_n) \in G''$ whose span intersects the box $Q_v$ we have
    \[
    x_{j}' \in \langle x_1, \ldots, x_{j}', \ldots, x_n\rangle(\delta) \subset \langle x_1, \ldots, v, \ldots, x_n\rangle(C\delta/\delta_0) \subset \langle x_1, \ldots, x_{j}, \ldots, x_n\rangle(C'\delta/\delta_0^2)  
    \]
    where we used that all 3 vectors $x'_j, v, x_j$ are $\gtrsim \delta_0$ separated from the flat $$\langle x_1, \ldots, x_{j-1}, x_{j+1}, \ldots, x_n\rangle.$$ Applying the thin-planes property for $\mu_j$, integrating over all tuples $(x_1, \ldots, \widehat{x_j}, \ldots, x_n) $ and summing over $j$, we conclude (\ref{eq:prod-Qv}).
    
    Now we observe
    \begin{align*}
    K' \delta^{\sigma-\varepsilon}(\prod \mu_i)(G_{\delta, K'}\cap G'') &\le \int_{G_{\delta, K'}\cap G''} \nu(\langle x_1, \ldots, x_n\rangle(\delta))d\mu_1(x_1)\dotsb d\mu_n(x_n) \\&\le \sum_v \nu(Q_v) \int_{G''} \mathbf{1}[\langle x_1, \ldots, x_n\rangle\cap Q_v\neq\emptyset] d\mu_1(x_1)\dotsb d\mu_n(x_n)\\
    &\lesssim \sum_v \nu(Q_v) K \delta_0^{-2} \delta^\sigma \lesssim K\delta_0^{-2}\delta^\sigma.
    \end{align*}
    so it follows that $G_{\delta, K'}\cap G''$ has measure at most $\lesssim (K\delta_0^{-2\sigma}/K')\delta^{\varepsilon}$. So taking $K'$ large enough we can make the sum over all dyadic $\delta$ less than $\varepsilon$. With that choice of $K'$ we get our desired graph $G' = G'' \cap \bigcap_{\delta \le \delta_0} G_{\delta, K'}$.
\end{proof}

\begin{lemma}\label{lem:projecting-NC-flats}
    Let $V_1, \ldots, V_m \subset \R^n$ be an {\bf NC} collection of flats. Fix $H \subset \R^n$ a generic hyperplane. Let $\Omega$ be the set of points $x \in \R^n\setminus H$ such that the collection of projected flats $\{ \pi_x(V_j\setminus \{x\}), ~ j\in [m] \}$ is not {\bf NC} inside $H$. Then we have $|\Omega| \le C$ for some constant $C=C(n, m)$.
\end{lemma}

\begin{proof}
    On the contrary, suppose that $|\Omega| > C(n,m)$ for a sufficiently large constant $C(n, m)$. For each partition $\mc I = \{I_1, \ldots, I_t\}$ of $[m]$ let $\Omega_{\mc I}$ be the set of $x \in \R^n\setminus H$ for which we have 
    \[
    \sum_{j=1}^t \dim \pi_x(V_{I_j}\setminus \{x\}) \le n-2.
    \]
    Clearly, $\Omega = \bigcup_{\mc I} \Omega_{\mc I}$. Now observe that by genericity of $H$, 
    \[
    \dim \pi_x(V_I) = \begin{cases}
        \dim V_I, ~~ x \not \in V_I, \\
        \dim V_I-1, ~ x \in V_I
    \end{cases}
    \]
    so if $x \in \Omega_{\mc I}$ for $\mc I = \{I_1, \ldots, I_t\}$ then we have
    \[
    n-2\ge \sum_{j=1}^t \dim V_{I_j} - 1_{x \in V_{I_j}} \ge n - \sum_{j=1}^t 1_{x \in V_{I_j}}
    \]
    where we used the {\bf NC} property of the collection $V_1, \ldots, V_m$. We conclude that there are some $j\neq j'$ such that $x \in V_{I_j} \cap V_{I_{j'}}$. Now this implies that if $|\Omega_{\mc I}|> t^2$ then there are some $x\neq x' \in \Omega_{\mc I}$ so that $x, x' \in V_{I_j} \cap V_{I_{j'}}$. This in turn implies that $\dim V_{I_j} \cap V_{I_{j'}} \ge 1$ and so 
    \[
    \dim \pi_x(V_{I_j\cup I_{j'}}\setminus \{x\}) \le \dim \pi_x(V_{I_j}\setminus \{x\})+ \dim \pi_x(V_{ I_{j'}}\setminus \{x\}).
    \]
    We conclude that $x \in \mc I'$ where $\mc I'$ is the partition of $[m]$ obtained from $\mc I$ by merging $I_j$ and $I_{j'}$. So by pigeonhole principle, we can find some $\mc I'$ with less parts than $\mc I$ such that $|\Omega_{\mc I'}| \gtrsim_m |\Omega_{\mc I}|$. So assuming that $C(n, m)$ is sufficiently large, by repeating this argument we eventually end up at the trivial partition $\{[m]\}$ of $[m]$ for which it is clear that $\Omega_{\{[m]\}} = \emptyset$, which is a contradiction.
\end{proof}

\begin{lemma}\label{lemma:projection-of-irreducible}
    Let $\mu$ be a measure on $V\cap B^n(0,1)$ such that $\mu$ is $(w, \tau)$-irreducible in $V$ for some $\tau \le 1/2$. Let $Q, U \subset \R^n$ be transversal subspaces with $\dim Q+\dim U = n-1$ and $d(Q, U) \ge c$ and consider the radial projection map $\pi_Q^U: \R^n\setminus Q\to U$. 
    
    Suppose that $V \not\subset U$ and $\dim\pi_Q^U(V)\ge 1$.

    Then for $\varepsilon\le w$ the measure $\pi_Q^U(\mu|_{V\setminus Q(\varepsilon)})$ is nonzero and $(cw, 2\tau)$-irreducible in $\pi_Q^U(V)$.
\end{lemma}

\begin{proof}
    For ease of notation, let $\pi = \pi_Q^U$. Since $\dim \pi(V)\ge 1$, $\pi(V)$ has nontrivial affine subspaces. 
Let $G\subset \pi(V)$ be a proper affine subspace, and note that since $G$ is proper, there exists $v\in V$ such that $\pi(v) \notin G$. 
Therefore, $v\notin \pi^{-1}(G)$, and hence $v\notin L = \pi^{-1}(G)\cap V$, so $L$ is a proper subflat of $V$.

By definition, $L\subset \pi^{-1}(G(\delta)) \cap V$, and furthermore, by the law of sines (using that $d(Q,U) \ge c$),
\[
\pi^{-1}(G(\delta))\cap V \subset L(c^{-1}\delta).
\]
Denote 
\[
\nu = \mu|_{V\setminus Q(\varepsilon)}.
\]
Since $\varepsilon \le w$, we first note that 
\[
\mu(Q(\varepsilon)\cap V) \le \mu((Q\cap V)(\varepsilon)) \le \tau\,\mu(B^n(0,1)),
\]
so 
\[
\nu(B^n(0,1)) = \mu(B^n(0,1)) - \mu(Q(\varepsilon)\cap V)
\ge (1-\tau)\,\mu(B^n(0,1)).
\]
Now, because $\mu$ is $(w,\tau)$-irreducible in $V$, we have for any proper affine subspace 
$G\subset \pi(V)$ and any $0<\delta\le c w$,
\[
\pi\nu(G(\delta)) 
= \nu(\pi^{-1}(G(\delta))) 
\le \nu(L(c^{-1}\delta)) 
\le \tau\,\mu(B^n(0,1)).
\]
Dividing both sides by $\nu(B^n(0,1))$ yields
\[
\frac{\pi\nu(G(\delta))}{\nu(B^n(0,1))} 
\le \frac{\tau}{1-\tau} 
\le 2\tau,
\]
provided $\tau\le \tfrac12$. 
Hence $\pi\nu$ is $(c w,2\tau)$-irreducible in $\pi(V)$.
\end{proof}

\newpage 

\appendix

\section{From irreducible measures to thin \texorpdfstring{$1$}{1}-planes}\label{appendix:thin}

In this appendix, we deduce the base case of Theorem \ref{theorem:section4main} from Ren's discretized radial projection result.

We recall the statement of Ren's discretized radial projection theorem. (Note that an $(r_0,k)$-plate $P$ is the $r_0$-neighborhood of a $k$-flat in $\R^n$, and that, moreover, such a plate is $\eta$-concentrated with respect to a measure $\mu$ if $\mu(P)\geq \eta$.)

\begin{theorem}[Theorem 1.13 of \cite{ren2023discretized}]\label{thm:ren-thin-tubes}
Let $k \in \{1,2,\dotsc,n-1\}$, $k-1<\sigma < s \le k$, and $\varepsilon>0$. There exist $N,K_0$ depending on $\sigma, s, k$, and $\eta(\varepsilon)>0$ (with $\eta(1) = 1$) such that the following holds. Fix $r_0 \le 1$, and $K \ge K_0$. Let $\mu_0,\mu_1$ be $\sim 1$-separated $s$-Frostman measures with constants $C_0,C_1$ supported on $X_0, X_1$, which lie in $B^n(0,1)$. Assume that $\|\mu_0\|,\|\mu_1\|\le 1$. Let $A$ be the pairs of $(x_0,x_1)\in X_0\times X_1$ that lie in some $K^{-1}$-concentrated $(r_0,k)$-plate. Then there exists a Borel set $B\subset X_0\times X_1$ with $(\mu_0\times \mu_1)(B) \lesssim K^{-\eta}$ such that for every $x_0\in X_0$ and $r$-tube $T$ containing $x_0$, we have
\[
\mu_1(T\setminus(A|_{x_0}\cup B|_{x_0})) \lesssim \frac{K^N}{r_0^{\sigma-(k-1)+N\varepsilon}}r^\sigma.
\]
The implicit constant may depend on $C_0,C_1,\sigma,s,k$.
\end{theorem}

We state the base case here for the convenience of the reader.

\begin{theorem}[Corollary of Theorem 1.13 in \cite{ren2023discretized}] 
\label{thm:renprojcorollary}
Let $0< s \leq n-1$, $\sigma < s$, and $C>0$. For every $\varepsilon > 0$, there exist $\tau_0 := \tau_0(\varepsilon, s, \sigma, n, C)$ and $w_0:= w_0(\varepsilon, s, \sigma, n, C)$ 
such that the following holds for all $\tau \le \tau_0$, and all $w \le w_0$.
Given $\mu_0,\mu_1$ that are $(w, \tau)$-irreducible $(C,s)$-Frostman measures that are $1/C$-separated in $\R^n$, there exists some $K:= K(\epsilon, s, \sigma, n,C, w) >0$ such that $(\mu_0,\mu_1)$ spans $(\sigma, K, 1-\epsilon)$-thin $1$-planes.
\end{theorem}

Having thin tubes is essentially equivalent to having thin 1-planes, up to harmless changes of constants. We will prove that thin tubes implies thin 1-planes with slightly worse constants. (The proof that thin $1$-planes implies thin tubes is not needed, and may safely be left to the reader.)

\begin{lemma}\label{lem:oswr_to_thin}
    For every $\varepsilon>0$ and $C>0$, there exists $A = A(\varepsilon,\sigma,n)$ and $B = B(\sigma)$ such that the following holds.
    Suppose that $\mu_0,\mu_1$ are measures supported in $X_0,X_1\subset B(0,1)$, and such that $\mathrm{dist}(X_0,X_1)\ge 1/C$. Suppose both $(\mu_0,\mu_1)$ and $(\mu_1,\mu_0)$ have $(\sigma,K,1-\varepsilon)$-thin tubes.
    Then $(\mu_0,\mu_1)$ have $(\sigma-\varepsilon, AK, 1-B\varepsilon)$-thin 1-planes.
\end{lemma}

\begin{proof}
    Let $T_{x_0, x_1}(r)$ denote $r$-tube around the line $x_0, x_1$. If $G\subset X_0\times X_1$ is the graph coming from the assumption $(\mu_0,\mu_1)$ have $(\sigma,K,1-\varepsilon)$-thin tubes, define
    $$
    E(r) = \{(x_0, x_1) \in G:~ \mu_1(T_{x_0, x_1}(r)) > 10\varepsilon^{-2} CK r^{\sigma-\varepsilon}\}.
    $$
    Fix $x_0 \in X_0$ and let us bound $\mu_1(E(r)|_{x_0})$. Let $I$ be a maximal subset of $E(r)|_{x_0}$ with the property that $\{\pi_{x_0}(x_1):x_1\in I\}$ is $r$-separated. Then 
    $$
    \mu_1(E(r)|_{x_0}) \le \sum_{x_1 \in I} \mu_1(G|_{x_0} \cap T_{x_0, x_1}(2r)) \le |I| K (2r)^\sigma 
    $$
    by the thin tubes property of $G$. Since $X_0$ and $X_1$ are $1/C$-separated, the intersections $T_{x_0, x_1}(r) \cap X_1$, $x_1 \in I$ are essentially disjoint (there is at most $10C$-wise overlap) and so
    $$
    10C = 10C\mu_1(X_1) \ge \sum_{x_1\in I} \mu_1(T_{x_0, x_1}(r)) \ge 10 C \varepsilon^{-2} K r^{\sigma-\varepsilon} |I|.
    $$
    We conclude that $|I| \le \varepsilon^2 K^{-1} r^{\varepsilon-\sigma}$. Thus, $\mu_1(E(r)|_{x_0}) \le 2^\sigma\varepsilon^2 r^{\varepsilon}$ holds for any $x_0 \in X_0$ and any dyadic $r > 0$. So if we define 
    $$
    E = \bigcup_{r \in 2^{-\N}} E(r)
    $$
    then it follows that 
    $$
    (\mu_0\times \mu_1)(E) \le 2^\sigma \sum_{r \in 2^{-\N}} \varepsilon^2 r^{\varepsilon} < C'\varepsilon.
    $$
    Hence $(\mu_0\times\mu_1)(G')\ge 1-C'\varepsilon$. If $(x_0,x_1)\in G'$, $r>0$, and $x_0,x_1\in T(r)$, then $T(r)\cap B(0,1)\subset T_{x_0,x_1}(r')$ for a dyadic $r'\le 10r$. Since $(x_0,x_1)\notin E(r')$, this implies 
    \[
    \mu_1(T(r))\le \mu_1(T_{x_0,x_1}(r'))\le 10C\varepsilon^{-2}Kr'^{\sigma-\varepsilon}\le 10^{1+\sigma-\varepsilon}C\varepsilon^{-2}Kr^{\sigma-\varepsilon}.
    \]
    With $G' = G\setminus E$, we have thus verified that $(\mu_0,\mu_1)$ have $(\sigma-\varepsilon, 10^{1+\sigma-\varepsilon}C\varepsilon^{-2}K, 1-C'\varepsilon)$-thin tubes with respect to $G'$.
\end{proof}

    In light of Lemma \ref{lem:oswr_to_thin}, it suffices to show that a pair of $\sim 1$-separated irreducible measures in $\mathbb R^n$ has thin tubes. We need one more lemma to obtain this result.

\begin{lemma}\label{thm:irreducible-implies-thin-tubes}
    Suppose that $\mu_0,\mu_1$ are $s\le n-1$-Frostman probability measures contained in $B^n(0,1)$ supported in $X_0, X_1$, respectively with $\mathrm{dist}(X_0,X_1)\ge 1/C$, and that they are irreducible in $\R^n$. Then for each $\varepsilon>0$ and $\sigma < s$, there exists $K(\varepsilon,s,\sigma,n)$ such that both $(\mu_0,\mu_1)$ and $(\mu_1,\mu_0)$ have $(\sigma, K,1-\varepsilon)$-thin tubes in the sense of Definition \ref{defn:thin-tubes}.
\end{lemma}

\begin{proof}
    Choose $k = n-1$ in the statement of Theorem \ref{thm:ren-thin-tubes}, and let $L\ge K_0$ be sufficiently large and $B$ be such that $(\mu_0\times\mu_1)(B) \lesssim L^{-\eta}\le  \varepsilon$. By Lemma \ref{lem:compactness-for-irreducible-measures}, for any $\tau>0$, there exists $w$ such that $\mu_0$ and $\mu_1$ are both $(w,\tau)$-irreducible. Thus, letting $\tau_0 = L^{-1}$, there exists $w_0>0$ such that $\mu_0$ and $\mu_1$ are both $(w_0,\tau_0)$-irreducible. Therefore, for all $w\le w_0 \le 1$, $\max\{\mu_0(V(w)),\mu_1(V(w))\} \le L^{-1}$.  In particular, the set $A$ of pairs $(x_0,x_1)\in X_0\times X_1$ lying in some $L^{-1}$-concentrated $(w_0,n-1)$-plate is empty. It follows that if $G = (X_0\times X_1)\setminus B$, then both $(\mu_0,\mu_1)$ and $(\mu_1,\mu_0)$ have $(\sigma,K,1-\varepsilon)$-thin tubes by Theorem \ref{thm:ren-thin-tubes} for some $K:= K(\epsilon, s, \sigma, n,C, w)$.    
\end{proof}

\printbibliography

\end{document}